\documentclass[numbook]{svjour3}
\usepackage{yuyuan}
\smartqed
\usepackage[utf8]{inputenc}
\usepackage{algorithm, algpseudocode, graphicx, braket, epstopdf}

\usepackage[colorlinks=true, citecolor=blue]{hyperref}
\usepackage{mathtools}
\mathtoolsset{showonlyrefs}
\usepackage{cellspace}
\usepackage{hhline}

\def\xl{\underline{x}}
\def\xu{\overline{x}}

\def\ul{\underline{u}}
\def\uu{\overline{u}}

\def\tu{\tilde{u}}
\def\tx{\tilde{x}}

\def\kp{\tp[k]}

\renewcommand{\epsilon}{\varepsilon}

\title{Accelerated gradient sliding for structured convex optimization
}
\author{
	Guanghui Lan
	\and
	Yuyuan Ouyang
	\thanks{The authors are partially supported by National Science Foundation grants 1319050, 1637473 and 1637474, and Office of Naval Research grant N00014-16-1-2802. }
}
\institute{
	Guanghui Lan\at {H. Milton Stewart School of Industrial and Systems Engineering, Georgia Institute of Technology
		\\
		\email{george.lan@isye.gatech.edu}}
	\and
	Yuyuan Ouyang\at{Department of Mathematical Sciences, Clemson University 
		\\
		\email{yuyuano@clemson.edu}}
}

\begin{document}
	
	\maketitle
	
	\begin{abstract}
	Our main goal in this paper is to show that one can skip gradient computations 
for gradient descent type methods applied to
certain structured convex programming (CP) problems.
To this end, we first present an accelerated gradient sliding (AGS) method for minimizing the summation of 
two smooth convex functions with different Lipschitz constants. We show that the AGS method can skip
the gradient computation for one of these smooth components without 
slowing down the overall optimal rate of convergence.
This result is much sharper than the classic
black-box CP complexity results especially when the difference between the two Lipschitz constants associated with
these components is large. 
We then consider an important class of bilinear saddle point problem whose objective function is given by
the summation of a smooth component and a nonsmooth one with a bilinear saddle point structure.
Using the aforementioned AGS method for smooth composite optimization 
and Nesterov's smoothing technique, we show that one only needs ${\cal O}(1/\sqrt{\epsilon})$ gradient computations
for the smooth component
while still preserving the optimal ${\cal O}(1/\epsilon)$ overall iteration complexity for solving
these saddle point problems. We demonstrate that
even more significant savings on gradient computations can be obtained   
for strongly convex smooth and bilinear saddle point problems.
	\end{abstract}

\section{Introduction}

In this paper, we show that one can skip gradient computations 
without slowing down the convergence of gradient descent type methods for solving
certain structured convex programming (CP) problems.
To motivate our study, let us first consider the following 
classic bilinear saddle point problem (SPP):
\begin{align}
	\label{eqSPP}
	\psi^*:=\min_{x\in X}\left\{ \psi(x):= f(x) + \max_{y\in Y}\langle Kx, y\rangle - J(y)\right\}.
\end{align}
Here, $X\subseteq\R^n$ and $Y\subseteq\R^m$ are closed convex sets, $K: \R^n \to \R^m$ is
a linear operator,
$J$ is a relatively simple convex function, and $f:X\to\R$ is a continuously differentiable convex function satisfying
\begin{align}
	\label{eqL}
	0 \le f(x) - l_f(u,x) \le \frac{L}{2}\|x - u\|^2,\ \forall x,u\in X,
\end{align}
for some $L>0$, 
where $l_f(u,x):=f(u) + \langle\nabla f(u), x-u\rangle$ denotes the first-order Taylor expansion of $f$ at $u$. 
Since $\psi$ is a nonsmooth convex function, traditional nonsmooth optimization methods, e.g., the subgradient method, would 
require $\cO(1/\varepsilon^2)$ iterations to find an $\varepsilon$-solution of \eqref{eqSPP}, i.e.,
a point $\bar x \in X$ s.t. $\psi(\bar x) - \psi^* \le \varepsilon$. In a landmarking work \cite{nesterov2005smooth}, 
Nesterov suggests to approximate $\psi$ by a smooth convex function
\begin{align}
	\label{eqSmoothApproxProblem}
	\psi^*_\rho:= \min_{x\in X} \left\{ \psi_\rho(x):=f(x) + h_\rho(x) \right\},
\end{align}
with
\begin{align}
	\label{eqhrho}
	h_\rho(x):=\max_{y\in Y}\langle Kx,y\rangle - J(y) - \rho W(y_0,y)
\end{align}
for some $\rho > 0$, where $y_0\in Y$ and $W(y_0,\cdot)$ is a strongly convex function.
By properly choosing $\rho$ and applying the optimal gradient method to \eqref{eqSmoothApproxProblem},
he shows that one can compute an $\varepsilon$-solution of \eqref{eqSPP} in at most
\begin{align}
\label{eqPrevOptBound}
\cO\left(\sqrt{\frac{L}{\varepsilon}} + \frac{\|K\|}{\varepsilon}\right)
\end{align}
iterations.
Following~\cite{nesterov2005smooth}, much research effort has been devoted to the 
development of first-order methods utilizing the saddle-point structure of \eqref{eqSPP} (see, e.g.,
the smoothing technique \cite{nesterov2005smooth,nesterov2005excessive,auslender2006interior,lan2011primal,d2008smooth,hoda2010smoothing,tseng2008accelerated,becker2011nesta,lan2013bundle},
the mirror-prox methods \cite{nemirovski2004prox,chen2014accelerated,he2013mirror,juditsky2011solving}, the primal-dual type methods \cite{chambolle2011first,tseng2008accelerated,esser2010general,zhu2008efficient,chen2014optimal,he2014accelerated} and their equivalent form as the alternating direction method of multipliers \cite{monteiro2013iteration,he2012convergence_siims,he2012convergence,ouyang2013stochastic,ouyang2015accelerated,he2013accelerating}). Some of these methods (e.g., \cite{chen2014optimal,chen2014accelerated,ouyang2015accelerated,he2014accelerated,lan2013bundle}) can achieve exactly the same complexity bound as in \eqref{eqPrevOptBound}.
 
One problem associated with Nesterov's smoothing scheme and the related methods mentioned above is that each iteration of 
these methods require both the computation of $\nabla f$ and the evaluation of the linear operators ($K$ and $K^T$).
As a result, the total number of gradient and linear operator evaluations will both be bounded by ${\cal O}(1/\epsilon)$. 
However, in many applications the computation of $\nabla f$ is often much
more expensive than the evaluation of the linear operators $K$ and $K^T$.
This happens, for example, when the linear operator $K$ is sparse (e.g., total variation,
overlapped group lasso and graph regularization), 
while $f$ involves a more expensive data-fitting term (see Section~\ref{secNumerical} and \cite{lan2015gradient}
for some other examples). 
In~\cite{lan2015gradient}, Lan considered some similar situation and
proposed a gradient sliding (GS) algorithm to minimize
a class of composite problems whose objective function is given by the summation 
of a general smooth and nonsmooth component. He shows that one can skip
the computation of the gradient for the smooth component from time to time, 
while still maintaining the ${\cal O}(1/\epsilon^2)$ iteration complexity bound.
More specifically, by applying the GS method in~\cite{lan2015gradient} to problem \eqref{eqSPP}, we can 
show that the number of gradient evaluations of $\nabla f$ will be bounded by
\begin{align}
\label{eqOptSmooth}
\cO\left(\sqrt\frac{L}{\varepsilon}\right),
\end{align}
which is significantly better than \eqref{eqPrevOptBound}. Unfortunately, the total number of evaluations for the linear operators $K$ 
and $K^T$ will be bounded by
\begin{align}
\label{eqSubGradGSIter}
\cO\left(\sqrt{\frac{L}{\varepsilon}} + \frac{\|K\|^2}{\varepsilon^2}\right),
\end{align}
which is much worse than \eqref{eqPrevOptBound}. 
An important yet unresolved research question is whether one can still preserve the optimal ${\cal O}(1/\epsilon)$
complexity bound in \eqref{eqPrevOptBound} for solving \eqref{eqSPP} by utilizing only ${\cal O}(1/\sqrt{\epsilon})$
gradient computations of $\nabla f$ to find an $\epsilon$-solution of \eqref{eqSPP}.
If so, we could be able to keep the total number of iterations relatively small,
but significantly reduce the total number of required gradient computations.


In order to address the aforementioned issues associated with existing solution methods for solving \eqref{eqSPP}, 
we pursue in this paper a different approach to exploit the structural information of \eqref{eqSPP}. Firstly,
instead of concentrating solely on nonsmooth optimization as in \cite{lan2015gradient},  
we study the following smooth composite optimization problem: 
\begin{align}
	\label{eqSmoothProblem}
	\phi^*:=\min_{x\in X}\left\{ \phi(x):=f(x) + h(x) \right\}.
\end{align}
Here $f$ and $h$ are smooth convex functions satisfying \eqref{eqL} and 
\begin{align}
	\label{eqM}
	0\le h(x)-l_h(u,x)\le \frac{M}{2}\|x-u\|^2,\ \forall x,u\in X,
\end{align}
respectively. It is worth noting that problem~\eqref{eqSmoothProblem} can be viewed 
as a special cases of \eqref{eqSPP} or \eqref{eqSmoothApproxProblem} (with $J=h^*$ being a strongly convex function, 
$Y=\R^n$, $K=I$ and $\rho=0$). Under the assumption that $M\ge L$, we present
a novel accelerated gradient sliding (AGS) method which can skip the computation
of $\nabla f$ from time to time. We show 
that the total number of required gradient evaluations of $\nabla f$ and $\nabla h$, respectively, 
can be bounded by  
\begin{align}
	\label{eqBoundLM}
\cO\left(\sqrt{\frac{L}{\varepsilon}}\right)	\ \ \ \mbox{and} \ \ \cO\left(\sqrt{\frac{M}{\varepsilon}}\right)
\end{align}
to find an $\varepsilon$-solution of \eqref{eqSmoothProblem}.
Observe that the above complexity bounds
are sharper than the complexity bound obtained by Nesterov's optimal method for 
smooth convex optimization, which is given by
\[
\cO\left(\sqrt{\frac{L+M}{\varepsilon}}\right).
\]
In particular, for the AGS method, the Lipschitz constant $M$ associated with $\nabla h$
does not affect at all the number of gradient evaluations of $\nabla f$. 
Clearly, the higher ratio of $M/L$ will potentially result in more savings on the gradient computation of $\nabla f$.
Moreover, if $f$ is strongly convex with modulus $\mu$, then the above two complexity
bounds in \eqref{eqBoundLM} can be significantly reduced to
\begin{align}
	\label{eqBoundLMs}
\cO\left(\sqrt{\frac{L}{\mu}} \log \frac{1}{\varepsilon}\right)	\ \ \ \mbox{and} \ \ \cO\left(\sqrt{\frac{M}{\mu}} \log \frac{1}{\varepsilon}\right),
\end{align} 
respectively, which also improves Nesterov's optimal method applied to \eqref{eqSmoothProblem}
in terms of the number gradient evaluations of $\nabla f$. 
Observe that in the classic black-box setting 
\cite{nemirovski1983problem,nesterov2004introductory}
the complexity bounds in terms of gradient evaluations of $\nabla f$ and $\nabla h$
are intertwined, and a larger Lipschitz constant $M$ will result in more gradient evaluations of $\nabla f$, 
even though there is no explicit relationship between $\nabla f$ and $M$.
In our development, we break down the black-box assumption
by assuming that we have separate access to $\nabla f$ and $\nabla h$ rather than $\nabla \phi$ as a whole.
To the best of our knowledge, these types of separate complexity bounds as in \eqref{eqBoundLM}
and \eqref{eqBoundLMs} have never been obtained before for smooth convex optimization.

Secondly, we apply the above AGS method to the smooth approximation problem~\eqref{eqSmoothApproxProblem} 
in order to solve the aforementioned bilinear SPP in~\eqref{eqSPP}. 
By choosing the smoothing parameter properly, we show that the total number
of gradient evaluations of $\nabla f$ and operator evaluations of $K$ (and $K^T$)
for finding an $\varepsilon$-solution of \eqref{eqSPP}
can be bounded by
\begin{align}
\label{eqOptsaddleL}
\cO\left(\sqrt\frac{L}{\varepsilon}\right) \ \ \ \mbox{and} \ \ \ \cO\left(\frac{\|K\|}{\varepsilon}\right),
\end{align}
respectively.  In comparison with Nesterov's original smoothing scheme and other existing
methods for solving \eqref{eqSPP}, our method can provide significant savings on the number of gradient computations of $\nabla f$
without increasing the complexity bound on the number of operator evaluations of $K$ and $K^T$. 
In comparison with the GS method in \cite{lan2015gradient}, our method
can reduce the number of operator evaluations of $K$ and $K^T$ from 
${\cal O}(1/ \epsilon^2)$ to ${\cal O}(1/\epsilon)$. Moreover, if $f$ is strongly convex with modulus $\mu$, 
the above two bounds will be significantly reduced to
\begin{align}
\label{eqOptsaddleSL}
\cO\left(\sqrt\frac{L}{\mu} \log \frac{1}{\epsilon}\right) \ \ \ \mbox{and} \ \ \ \cO\left(\frac{\|K\|}{\sqrt{\varepsilon}}\right),
\end{align}
respectively.
To the best of our knowledge, this is the first time that these tight complexity bounds
were obtained for solving the classic bilinear saddle point problem \eqref{eqSPP}.

It should be noted that, even though the idea of skipping the computation of $\nabla f$
is similar to \cite{lan2015gradient}, the AGS method presented in this paper significantly
differs from the GS method in  \cite{lan2015gradient}. In particular, each iteration of GS method
consists of one accelerated gradient iteration together with a bounded number of subgradient iterations.
On the other hand, each iteration of the AGS method is composed of an accelerated
gradient iteration nested with a few other accelerated gradient iterations to solve a
different subproblem. The development of the AGS method seems to
be more technical than GS and its convergence analysis is also highly nontrivial.

This paper is organized as follows. We first present the AGS method and discuss its convergence properties 
for minimizing the summation of two smooth convex functions \eqref{eqSmoothProblem} in Section \ref{secAGSS}. 
Utilizing this new algorithm and its associated convergence results, we study the properties of the AGS method 
for minimizing the bilinear saddle point problem \eqref{eqSPP} in Section \ref{secSPP}. 
We then demonstrate the effectiveness of the AGS method through out preliminary numerical
experiments for solving certain portfolio optimization and image reconstruction problems in Section \ref{secNumerical}.
Some brief concluding remarks are made in Section \ref{secConclusion}.

\subsection*{Notation, assumption and terminology}
We use $\|\cdot\|$, $\|\cdot\|_*$, and $\langle\cdot,\cdot\rangle$ to denote an arbitrary norm, the associated conjugate and the inner product
in Euclidean space $\R^n$, respectively. 
For any set $X$, we say that $V(\cdot,\cdot)$ is a prox-function associated with $X \subseteq \R^n$ modulus $\nu$ if there exists a 
strongly convex function $\pi(\cdot)$ with strong convexity parameter $\nu$ such that
\begin{align}
	\label{eqV}
	\begin{aligned}
	V(x,u)=&\pi(u) - \pi(x) - \langle\nabla\pi(x), u-x\rangle,\ \forall x,u\in X.
	\end{aligned}
\end{align}
The above prox-function is also known as the Bregman divergence \cite{bregman1967relaxation} (see also \cite{nesterov2005smooth,nemirovski2004prox}), 
which generalizes the Euclidean distance $\|x-u\|_2^2/2$. It can be easily seen from \eqref{eqV} and the strong convexity of $\pi$ that 
\begin{align}
\label{eqVnorm}
V(x,u)\ge \frac{\nu}{2}\|x-u\|^2 \ \ \forall x, y \in X.
\end{align}
Moreover, we say that the prox-function grows quadratically if there exists a constant $C$ such that $V(x,u)\le C\|x-u\|^2/2$. Without loss of generality, 
we assume that $C=1$ whenever this happens, i.e.,
\begin{align}
	\label{eqVquad}
	V(x,u)\le\frac{1}{2}\|x-u\|^2.
\end{align}
In this paper, we associate sets $X\subseteq \R^n$ and $Y \subseteq \R^m$ with prox-functions $V(\cdot,\cdot)$ and $W(\cdot,\cdot)$ with moduli $\nu$ and $\omega$
w.r.t. their respective norms in $\R^n$ and $\R^m$.

For any real number $r$, $\lceil r\rceil$ and $\lfloor r\rfloor$ denote the nearest integer to $r$ from above and below, respectively. We denote the set of nonnegative and positive real numbers by $\R_+$ and $\R_{++}$, respectively.

\section{Accelerated gradient sliding for composite smooth optimization}
\label{secAGSS}
In this section, we present an accelerated gradient sliding (AGS) algorithm for solving the smooth composite
optimization problem in \eqref{eqSmoothProblem} and discuss its convergence properties.  Our main objective
is to show that the AGS algorithm can skip the evaluation of $\nabla f$ from time to time and achieve better complexity bounds in terms of gradient computations than the classical optimal 
first-order methods applied to \eqref{eqSmoothProblem} (e.g., Nesterov's method in \cite{nesterov1983method}).
Without loss of generality, throughout this section we assume that $M\ge L$ in \eqref{eqL} and \eqref{eqM}. 



The AGS method evolves from the gradient sliding (GS) algorithm in \cite{lan2015gradient}, which was designed
to solve a class of composite convex optimization problems with the
objective function given by the summation of a smooth and nonsmooth component.
The basic idea of the GS method is to keep the nonsmooth term inside the projection (or proximal mapping)
in the accelerated gradient method and then to
apply a few subgradient descent iterations to solve the projection
subproblem. 
Inspired by \cite{lan2015gradient}, we suggest to keep the 
smooth term $h$ that has a larger Lipschitz constant in the proximal mapping
in the accelerated gradient method, and then to apply a few accelerated
gradient iterations to solve this smooth subproblem.
As a consequence, the proposed AGS method
involves two nested loops (i.e., outer and inner iterations), each of which consists
of a set of modified accelerated gradient descent iterations (see Algorithm \ref{algAGS}). 
At the $k$-th outer iteration, we first build a linear approximation $g_k(u) = l_f(\xl_k,u)$ of $f$ at the search point  $\xl_k \in X$
and then call the ProxAG procedure in \eqref{eqAGSoutput} to compute a new pair of search points
$(x_k, \tx_k) \in X \times X$. The ProxAG procedure can be viewed as a subroutine
to compute a pair of approximate solutions to
\begin{align} \label{inner_problem}
\min_{u\in X}g_k(u)+ h(u) + \beta V(x_{k-1},u),
\end{align}
where $g_k(\cdot)$ is defined in \eqref{eqgk}, and $x_{k-1}$ is called the prox-center at the $k$-th outer iteration.
It is worth mentioning that there are two essential differences associated with the steps \eqref{eqxl}-\eqref{equl}
from the standard Nesterov's accelerated gradient iterations.
Firstly, we use two different search points, i.e., $x_k$ and $\xu_k$, respectively, to update
$\xl_k$ to compute the linear approximation and $\xu_k$
to compute the output solution in \eqref{eqxu}.
Secondly, we employ two parameters, i.e., $\gamma_k$ and $\lambda_k$,
to update $\xl_k$ and $\xu_k$, respectively, rather than just one single parameter. 

The ProxAG procedure in Algorithm~\ref{algAGS} performs $T_k$ inner accelerated gradient iterations to solve \eqref{inner_problem}
with certain properly chosen starting points $\tu_0$ and $u_0$.  It should be noted, however, that
the accelerated gradient iterations in \eqref{equl}-\eqref{eqtu}  also differ from the standard Nesterov's accelerated 
gradient iterations in the sense that the definition of the search point  $\ul_t$ involves a fixed search point $\xu$.
Since each inner iteration of the ProxAG procedure requires one evaluation of $\nabla h$ and
no evaluation of $\nabla f$, the number of gradient evaluations of $\nabla h$ will be greater 
than that of $\nabla f$ as long as $T_k > 1$.
On the other hand, if $\lambda_k\equiv \gamma_k$ and  $T_k\equiv1$  in the AGS method, 
and $\alpha_t\equiv 1$, and $p_t\equiv q_t\equiv 0$ in the ProxAG procedure, then \eqref{eqAGSoutput} becomes
\begin{align}
x_k=\tx_k = \argmin_{u\in X}g_k(u) + l_h(\xl_k,u) + \beta_k V(x_{k-1}, u).
\end{align}
In this case, the AGS method reduces to a variant of 
Nesterov's optimal gradient method (see, e.g., \cite{nesterov2004introductory,tseng2008accelerated}).

\begin{algorithm}
	\caption{\label{algAGS}Accelerated gradient sliding (AGS) algorithm for solving \protect\eqref{eqSmoothProblem}}
	\begin{algorithmic}
		\State Choose $x_0\in X$. Set $\xu_0=x_0$.
		\For {$k=1,\ldots,N$}
		\begin{align}
		\label{eqxl}
		\xl_k=& (1-\gamma_k)\xu_{k-1} + \gamma_k x_{k-1},\\
		g_k(\cdot)=&l_f(\xl_k,\cdot), \label{eqgk}\\
		\label{eqAGSoutput}
		(x_k, \tx_k) = & ProxAG (g_k, \xu\kp, x\kp, \lambda_k, \beta_k, T_k)
		\\
		\label{eqxu}
		\xu_k = & (1-\lambda_k)\xu_{k-1} + \lambda_k \tx_k.
		\end{align}
		\EndFor
		\State Output $\xu_N$.
		
		\vgap
		
		\Procedure{$(x^+, \tx^+) = ProxAG$}{$g, \xu, x, \lambda, \beta, \gamma, T$}
		\State Set $\tu_0 = \xu$ and $u_0 = x$.
		\For {$t=1,\ldots, T$}
		\begin{align}
		\label{equl}
		\ul_t & = (1-\lambda)\xu + \lambda(1-\alpha_t)\tu\tp + \lambda\alpha_tu\tp,
		\\
		\label{equt}
		u_t & = \argmin_{u\in X}g(u) + l_h(\ul_t,u) + {\beta}V(x,u) + (\beta p_t+q_t) V(u\tp, u),
		\\
		\label{eqtu}
		\tu_t & = (1-\alpha_t)\tu\tp + \alpha_tu_t,
		\end{align}
		\EndFor
		\State Output $x^+=u_T$ and $\tx^+ = \tu_T$.
		\EndProcedure
	\end{algorithmic}
\end{algorithm}

Our goal in the remaining part of this section is to
establish the convergence of the AGS method and to
provide theoretical guidance to specify quite a few parameters,
including $\{\gamma_k\}$,
$\{\beta_k\}$, $\{T_k\}$, $\{\lambda_k\}$, $\{\alpha_t\}$, $\{p_t\}$, and $\{q_t\}$, used
in the generic statement of this algorithm.
In particular, we will provide upper bounds on the number of
outer and inner iterations, corresponding to 
 the number of gradient evaluations of $\nabla f$ and $\nabla h$, respectively,
 performed by the AGS method to 
find an $\varepsilon$-solution to \eqref{eqSmoothProblem}. 

We will first  study the convergence properties of the ProxAG procedure from which
the convergence of the AGS method immediately follows.
In our analysis, we measure the quality of the output solution computed at the $k$-th call to the ProxAG procedure by 
\begin{align}
\label{eqQk}
Q_k(x,u):=&g_k(x) - g_k(u) + h(x) - h(u).
\end{align}
Indeed, if $x^*$ is an optimal solution to \eqref{eqSmoothProblem}, then $Q_k(x,x^*)$ provides
a linear approximation for the functional optimality gap $\phi(x) - \phi(x^*)=f(x)-f(x^*)+h(x)-h(x^*)$ 
obtained by replacing $f$ with $g_k$. The following result describes some relationship between $\phi(x)$ and $Q_k(\cdot,\cdot)$.

\vgap
\begin{lemma}
	\label{lemPsiQ}
	For any $u\in X$, we have
	\begin{align}
	\label{eqPsiQ}
\begin{aligned}
	& \phi(\xu_k) - \phi(u) 
	\\
	\le & (1-\gamma_k)[\phi(\xu\kp) - \phi(u)] + Q_k(\xu_k, u) - (1-\gamma_k)Q_k(\xu\kp, u) 
	\\
	& + \frac{L}{2}\|\xu_k - \xl_k\|^2.
\end{aligned}	\end{align}
\end{lemma}

	\begin{proof}
		By \eqref{eqL}, \eqref{eqSmoothProblem}, \eqref{eqgk}, and the convexity of $f(\cdot)$, we have
		\begin{align}
		& \phi(\xu_k) - (1-\gamma_k)\phi(\xu\kp) - \gamma_k\phi(u)
		\\
		\le & g_k(\xu_k) + \frac{L}{2}\|\xu_k - \xl_k\|^2  + h(\xu_k) 
		\\
		& - (1-\gamma_k)f(\xu\kp) - (1-\gamma_k)h(\xu\kp) - \gamma_kf(u) -  \gamma_kh(u)
		\\
		\le & g_k(\xu_k) + \frac{L}{2}\|\xu_k - \xl_k\|^2  + h(\xu_k) 
		\\
		& - (1-\gamma_k)g_k(\xu\kp) - (1-\gamma_k)h(\xu\kp) - \gamma_kg_k(u) -  \gamma_kh(u)
		\\
		= & Q_k(\xu_k, u) - (1-\gamma_k)Q_k(\xu\kp, u) + \frac{L}{2}\|\xu_k - \xl_k\|^2.
		\end{align}
	\qed\end{proof}
	
\vgap

In order to analyze the ProxAG procedure, we need the following two technical results. The first one below characterizes the solution of optimization problems 
involving prox-functions. The proof of this result can be found, for example, in Lemma 2 of \cite{ghadimi2012optimal}.

\vgap 

\begin{lemma}
	\label{lemProx}
	Suppose that a convex set $Z\subseteq\R^n$, a convex function $q:Z\to\R$, points $z,z'\in Z$ and scalars $\mu_1,\mu_2\in\R_+$ are given.
	Also let $V(z,u)$ be a prox-function. If
	\begin{align}
	u^*\in\Argmin_{u\in Z} q(u) + \mu_1V(z, u) + \mu_2V(z', u),
	\end{align}
	then for any $u\in Z$, we have
	\begin{align}
	& q(u^*) + \mu_1V(z, u^*) + \mu_2V(z', u^*) 
	\\
	\le & q(u) + \mu_1V(z, u) + \mu_2V(z', u) - (\mu_1+\mu_2)V(u^*,u).
	\end{align}
\end{lemma}

\vgap 

The second technical result slightly generalizes Lemma 3 of \cite{lan2013bundle} to provide a convenient way to study sequences with sublinear rates of convergence.

\vgap

\begin{lemma}
	\label{lemSum}
	Let $c_k\in (0,1]$, $k=1,2,\ldots$ and $C_1>0$ be given, and define 
	\begin{align}
	\label{eqW}
	C_k:=(1-c_k)C\kp,\ k\ge 2.
	\end{align}
	Suppose that $C_k>0$ for all $k\ge 2$ and that the sequence $\{\delta_k\}_{k\ge 0}$ satisfies
	\begin{align}
	\label{eqDeltaCond}
	\delta_k\le (1-c_k) \delta\kp + B_k,\ k=1,2,\ldots.
	\end{align}
	For any $k\ge 1$, we have
	\begin{align}
	\label{eqDeltaBound}
	\delta_k\le C_k\left[\frac{1-c_1}{C_1}\delta_0 + \sum_{i=1}^k\frac{B_i}{C_i} \right].
	\end{align}
	In particular, the above inequality becomes equality when the relations in \eqref{eqDeltaCond} are all equality relations.
\end{lemma}

	\begin{proof}
		The result follows from dividing both sides of \eqref{eqDeltaCond} by $C_k$ and then summing up the resulting inequalities or equalities.
	\qed\end{proof}

\vgap

It should be noted that, although \eqref{eqDeltaCond} and \eqref{eqDeltaBound} are stated in the form of inequalities,
we can derive some useful formulas by setting them to be equalities. 
For example, let $\{\alpha_t\}$ be the parameters used in the ProxAG procedure (see \eqref{equl} and \eqref{eqtu}) and 
consider the sequence $\{\Lambda_t\}_{t\ge 1}$ defined by
\begin{align}
\label{eqLambda}
\Lambda_t = \begin{cases}
1 & t=1,
\\
(1-\alpha_t)\Lambda\tp & t>1.
\end{cases}
\end{align}
By Lemma \ref{lemSum} (with $k=t$, $C_k=\Lambda_t$, $c_k=\alpha_t$, $\delta_k\equiv 1$, and $B_k=\alpha_t$), we have
\begin{align}
\label{eqLambdaConv}
1 = \Lambda_t\left[\frac{1-\alpha_1}{\Lambda_1} + \sumt\frac{\alpha_i}{\Lambda_i}\right] = \Lambda_t(1-\alpha_1) + \Lambda_t\sumt\frac{\alpha_i}{\Lambda_i},
\end{align}
where the last identity follows from the fact that $\Lambda_1=1$ in \eqref{eqLambda}.
Similarly, applying Lemma \ref{lemSum} to the recursion $\tu_t  = (1-\alpha_t)\tu\tp + \alpha_tu_t$ in
\eqref{eqtu} (with $k=t$, $C_k=\Lambda_t$, $c_k=\alpha_t$, $\delta_k=\tu_t$, and $B_k=\alpha_tu_t$), we have
\begin{align}
\label{eqtuConv}
\tu_t = \Lambda_t\left[(1-\alpha_1)\tu_0 + \sumt\frac{\alpha_i}{\Lambda_i}u_i\right].
\end{align}
In view of \eqref{eqLambdaConv} and the fact that $\tu_0 = \xu$ in the description of the ProxAG procedure, 
the above relation indicates that $\tu_t$ is a convex combination of $\xu$ and $\{u_i\}_{i=1}^t$.

\vgap

With the help of the above two technical results, we are now ready to derive 
some important convergence properties for the ProxAG procedure in terms of
the error measure $Q_k(\cdot,\cdot)$. For the sake of notational convenience,
when we work on the $k$-th call to the ProxAG procedure, we drop the subscript $k$ in \eqref{eqQk}
and just denote 
\begin{align}
	\label{eqQ}
	Q(x,u):=g(x) - g(u) + h(x) - h(x). 
\end{align}
In a similar vein, we also
define 
\begin{align}
\label{eqxlxu}
\xl:=(1-\gamma)\xu + \gamma x \text{ and }\xu^+:=(1-\lambda)\xu + \lambda \tx^+.
\end{align}
Comparing the above notations with \eqref{eqxl} and \eqref{eqxu}, we can observe 
that $\xl$ and $\xu^+$, respectively, represent $\xl_k$ and $\xu_k$ in the $k$-th call to the ProxAG procedure. 

\vgap

\begin{lemma}
	\label{lemuOC}
	Consider the $k$-th call to the ProxAG procedure in Algorithm \ref{algAGS} and let
	$\Lambda_t$ and $\xu^+$ be defined in \eqref{eqLambda} and \eqref{eqxlxu} respectively. 
	If the parameters satisfy
	\begin{align}
	\label{eqGScond}
	\lambda\le 1, \Lambda_T(1-\alpha_1) = 1 - \frac{\gamma}{\lambda},\text{ and } \beta p_t + q_t \ge \frac{\lambda M\alpha_t}{\nu},
	\end{align}
	then
	\begin{align}
	\label{eqQBound}
	\begin{aligned}
	& Q(\xu^+, u) - (1-\gamma)Q(\xu, u) \le \Lambda_T\sum_{t=1}^T\frac{\Upsilon_t(u)}{\Lambda_t},\ \forall u\in X,
	\end{aligned}
	\end{align}
	where 
	\begin{align}
	\label{eqUpsilon}
	\Upsilon_t(u):= & {\lambda\beta\alpha_t}[V(x,u) - V(x,u_t) + p_tV(u\tp[t], u) - (1+p_t)V(u_t,u)]
	\\
	& + {\lambda\alpha_tq_t}[V(u\tp[t], u) - V(u_t,u)].
	\end{align}
	\end{lemma}

\begin{proof}
		Let us fix any arbitrary $u\in X$ and denote
		\begin{align}
		\label{eqv}
		v := (1-\lambda)\xu + \lambda u, \text{ and } \uu_t:=(1-\lambda)\xu + \lambda \tu_t.
		\end{align} 
		Our proof consists of two major parts. We first prove that
		\begin{align}
			\label{eqPart1}
			Q(\xu^+, u) - (1-\gamma)Q(\xu, u) \le Q(\uu_T, v) - \left(1 - \frac{\lambda}{\gamma}\right)Q(\uu_0, v),
		\end{align}
		and then estimate the right-hand-side of \eqref{eqPart1} through the following recurrence property:
		\begin{align}
			\label{eqPart2}
			Q(\uu_t, v) - (1-\alpha_t)Q(\uu\tp, v)\le \Upsilon_t(u).
		\end{align}
		The result in \eqref{eqQBound} then follows as an immediate consequence of \eqref{eqPart1} and \eqref{eqPart2}. 
		Indeed, by Lemma~\ref{lemSum} applied to \eqref{eqPart2} (with $k=t$, $C_k=\Lambda_t$, $c_k=\alpha_t$, $\delta_k=Q(\uu_t,v)$, and $B_k=\Upsilon_t(u)$), 
		we have
		\begin{align}
			Q(\uu_T, v) \le &\Lambda_T\left[\frac{1-\alpha_1}{\Lambda_1}Q(\uu_0,v) - \sum_{t=1}^T\frac{\Upsilon_t(u)}{\Lambda_t}\right] 
			\\
			= & \left(1 - \frac{\lambda}{\gamma}\right)Q(\uu_0,v) - \Lambda_T\sum_{t=1}^T\frac{\Upsilon_t(u)}{\Lambda_t},
		\end{align}
		where last inequality follows from \eqref{eqGScond} and the fact that $\Lambda_1=1$ in \eqref{eqLambda}. 
		The above relation together with \eqref{eqPart1} then clearly imply \eqref{eqQBound}.
		
		We start with the first part of the proof regarding \eqref{eqPart1}. By \eqref{eqQ} and the linearity of $g(\cdot)$, we have
		\begin{align}
		\label{eqtmpQxpx}
		\begin{aligned}
			& Q(\xu^+, u) - (1-\gamma)Q(\xu, u)
			\\
			= & g(\xu^+ - (1-\gamma)\xu - \gamma u) + h(\xu^+) - (1-\gamma)h(\xu) - \gamma h(u)
			\\
			= & g(\xu^+ - \xu + \gamma(\xu - u)) + h(\xu^+) - h(\xu) + \gamma(h(\xu) - h(u)).
		\end{aligned}
		\end{align}
		Now, noting that by the relation between $u$ and $v$ in \eqref{eqv}, we have
		\begin{align}
			\label{eqtmpxuu}
			\gamma(\xu - u) = \frac{\gamma}{\lambda}(\lambda\xu - \lambda u) = \frac{\gamma}{\lambda}(\xu-v).
		\end{align}
		In addition, by \eqref{eqv} and the convexity of $h(\cdot)$, we obtain
		\begin{align}
		\frac{\gamma}{\lambda}[h(v) - (1-\lambda)h(\xu) - \lambda h(u)]\le 0,
		\end{align}
		or equivalently,
		\begin{align}
			\label{eqtmphxuu}
		\gamma(h(\xu)-h(u))\le \frac{\gamma}{\lambda}(h(\xu)-h(v)).
		\end{align}
		Applying \eqref{eqtmpxuu} and \eqref{eqtmphxuu} to \eqref{eqtmpQxpx}, and using the definition of $Q(\cdot,\cdot)$ in \eqref{eqQ}, we obtain 
		\begin{align}
			& Q(\xu^+, u) - (1-\gamma)Q(\xu, u) \le Q(\xu^+, v) - \left(1 - \frac{\lambda}{\gamma}\right)Q(\xu, v).
		\end{align}
		Noting that $\tu_0=\xu$ and $\tx = \tu_T$ in the description of the ProxAG procedure, by \eqref{eqxlxu} and \eqref{eqv} we have $\xu^+ = \uu_T$ and $\uu_0=\xu$. Therefore, the above relation is equivalent to \eqref{eqPart1}, and we conclude the first part of the proof.
		
		For the second part of the proof regarding \eqref{eqPart2}, first observe that by the definition of $Q(\cdot,\cdot)$ in \eqref{eqQ}, the convexity of $h(\cdot)$, and \eqref{eqM},
		\begin{align}
		\label{tempRel}
		\begin{aligned}
		& Q(\uu_t, v) - (1-\alpha_t)Q(\uu\tp, v)
		\\
		= & \lambda\alpha_t(g(u_t) - g(u)) + h(\uu_t) - (1-\alpha_t)h(\uu\tp) - \alpha_th(v)
		\\
		\le & \lambda\alpha_t(g(u_t) - g(u)) + l_h(\ul_t, \uu_t) + \frac{M}{2}\|\uu_t - \ul_t\|^2 
		\\
		&- (1-\alpha_t)l_h(\ul_t,\uu\tp) - \alpha_tl_h(\ul_t,v)
		\\
		= & \lambda\alpha_t(g(u_t) - g(u)) + l_h(\ul_t, \uu_t-(1-\alpha_t)\uu\tp- \alpha_tv) + \frac{M}{2}\|\uu_t - \ul_t\|^2. 
		\end{aligned}
		\end{align}
		Also note that by \eqref{equl}, \eqref{eqtu}, and \eqref{eqv},
		\begin{align}
		& \uu_t - (1-\alpha_t)\uu\tp - \alpha_t v 
		= (\uu_t - \uu\tp) + \alpha_t(\uu\tp - v) 
		\\
		=&  \lambda(\tu_t - \tu\tp) + \lambda\alpha_t(\tu\tp - u)
		=  \lambda(\tu_t - (1-\alpha_t)\tu\tp) - \lambda\alpha_t u 
		\\
		=& \lambda\alpha_t(u_t - u).
		\end{align}
		By a similar argument as the above, we have
		\begin{align}
		\label{equublb}
		\uu_t - \ul_t = \lambda(\tu_t - (1-\alpha_t)\tu\tp) - \lambda\alpha_tu\tp = \lambda\alpha_t(u_t - u\tp).
		\end{align}
		Using the above two identities in \eqref{tempRel}, we have
		\begin{align}
		& Q(\uu_t, v) - (1-\alpha_t)Q(\uu\tp, v)
		\\
		\le & \lambda\alpha_t\left[g(u_t)-g(u)   + l_h(\ul_t, u_t) - l_h(\ul_t, u) + \frac{M\lambda\alpha_t}{2}\|u_t - u\tp\|^2\right].
		\end{align}
		Moreover, it follows from Lemma \ref{lemProx} applied to \eqref{equt} that
		\begin{align}
		\begin{aligned}
		& g(u_t)-g(u)   + l_h(\ul_t, u_t) - l_h(\ul_t, u)
		\\
		\le & \beta(V(x,u) - V(u_t, u) - V(x, u_t))
		\\
		&+ (\beta p_t+q_t)(V(u\tp, u) - V(u_t, u) - V(u\tp, u_t)).
		\end{aligned}
		\end{align}
		Also by \eqref{eqVnorm} and \eqref{eqGScond}, we have
		\begin{align}
			\frac{M\lambda\alpha_t}{2}\|u_t - u\tp\|^2 \le \frac{M\lambda\alpha_t}{2\nu}V(u\tp, u_t) \le (\beta p_t+q_t)V(u\tp, u_t).
		\end{align}
		Combining the above three relations, we conclude \eqref{eqPart2}.
	\qed\end{proof}

\vgap

In the following proposition, we provide certain sufficient conditions under which the  the right-hand-side of \eqref{eqQBound} 
can be properly bounded. As a consequence, we obtain a recurrence relation for the ProxAG procedure in terms of $Q(\xu_k,u)$.

\vgap

\begin{proposition}
	\label{proQBoundSpec}
	Consider the $k$-th call to the ProxAG procedure. If \eqref{eqGScond} holds,
	\begin{align}
	\label{eqAGSpar}
	\frac{\alpha_tq_t}{\Lambda_t} = \frac{\alpha\tn  q\tn}{\Lambda\tn} \ \text{ and } \ \frac{\alpha_t(1+p_t)}{\Lambda_t} = \frac{\alpha\tn  p\tn}{\Lambda\tn}
	\end{align}
	for any $1 \le t \le T-1$, then we have
	\begin{align}
	\label{eqQBoundSpec}
	\begin{aligned}
	& Q(\xu^+, u) - (1-\gamma)Q(\xu, u) 
	\\
	\le & \lambda\alpha_T[\beta(1+p_T)+q_T]\left[V(x,u) - V(x^+,u)\right] -  \frac{\nu\beta}{2\gamma}\|\xu^+ - \xl\|^2,
	\end{aligned}	
	\end{align}
	where $\xu^+$ and $\xl$ are defined in \eqref{eqxlxu}.
	\end{proposition}

\begin{proof}
		To prove the proposition it suffices to estimate the right-hand-side of \eqref{eqQBound}. We make three observations regarding the terms in \eqref{eqQBound} and \eqref{eqUpsilon}. First, by \eqref{eqLambdaConv},
		\begin{align}
		\lambda\beta\Lambda_T\sum_{t=1}^{T}\frac{\alpha_t}{\Lambda_t}V(x,u) = \lambda\beta(1-\Lambda_T(1-\alpha_1))V(x,u).
		\end{align}
		Second, by \eqref{eqVnorm}, \eqref{eqLambdaConv}, \eqref{eqtuConv}, \eqref{eqGScond}, and the fact that $\tu_0=\xu$ and $\tx^+=\tu_T$ in the ProxAG procedure, we have
		\begin{align}
		\lambda\beta\Lambda_T\sum_{t=1}^{T}\frac{\alpha_t}{\Lambda_t}V(x,u_t) \ge& \frac{\nu\gamma\beta}{2}\cdot\frac{\Lambda_T}{(1-\Lambda_T(1-\alpha_1))}\sum_{t=1}^{T}\frac{\alpha_t}{\Lambda_t}\|x - u_t\|^2
		\\
		\ge & \frac{\nu\gamma\beta}{2}\left\|x-\frac{\Lambda_T}{1-\Lambda_T(1-\alpha_1)}\sumt[T]\frac{\alpha_t}{\Lambda_t}u_t\right\|^2
		\\
		= & \frac{\nu\gamma\beta}{2}\left\|x-\frac{\tu_T - \Lambda_T(1-\alpha_1)\tu_0}{1-\Lambda_T(1-\alpha_1)} \right\|^2
		\\
		= & \frac{\nu\gamma\beta}{2}\left\|x-\frac{\lambda}{\gamma}\tu_T - \left(1-\frac{\lambda}{\gamma}\right)\tu_0 \right\|^2
		\\
		= & \frac{\nu\beta}{2\gamma}\left\|\gamma x- \lambda\tx^+ - (\gamma-\lambda)\xu\right\|^2
		\\
		= & \frac{\nu\beta}{2\gamma}\|\xl - \xu^+\|^2,
		\end{align}
		where the last equality follows from \eqref{eqxlxu}.		
		Third, by \eqref{eqAGSpar}, the fact that $\Lambda_1=1$ in \eqref{eqLambda}, and
		the relations that $u_0 = x$ and $u_T=x^+$ in the ProxAG procedure, we have
		\begin{align}
		&\lambda\beta\Lambda_T\sum_{t=1}^{T}\frac{\alpha_t}{\Lambda_t}[p_tV(u\tp[t], u) - (1+p_t)V(u_t,u)] 
		\\
		&+ \lambda\Lambda_T\sum_{t=1}^{T}\frac{\alpha_tq_t}{\Lambda_t}[V(u\tp[t], u) - V(u_t,u)]
		\\
		=& \lambda\beta\Lambda_T\left[\alpha_1p_1V(u_0,u) - \sumt[T-1]\left(\frac{\alpha_t(1+p_t)}{\Lambda_t} - \frac{\alpha\tn[t]p\tn[t]}{\Lambda\tn[t]}\right)V(u_t,u) \right.
		\\
		& \left.- \frac{\alpha_T(1+p_T)}{\Lambda_T}V(u_T,u)\right] + \lambda\alpha_Tq_T[V(u_0, u) - V(u_T,u)]
		\\
		= & \lambda\beta\left[\Lambda_T\alpha_1p_1V(u_0,u) - \alpha_T(1+p_T)V(u_T,u)\right] +  \lambda\alpha_Tq_T[V(u_0, u) - V(u_T,u)]
		\\
		= & \lambda\beta\left[\Lambda_T\alpha_1p_1V(x,u) - \alpha_T(1+p_T)V(x^+,u)\right] +  \lambda\alpha_Tq_T[V(x, u) - V(x^+,u)]	
		.
		\end{align}
		Using the above three observations in \eqref{eqQBound}, we have
		\begin{align}
		\begin{aligned}
			& Q(\xu^+, u) - (1-\gamma)Q(\xu, u)
			\\
			\le & \lambda\beta\left[(1-\Lambda_T(1-\alpha_1) + \Lambda_T\alpha_1p_1)V(x,u) - \alpha_T(1+p_T)V(x^+,u)\right] 
			\\
			& + \lambda\alpha_Tq_T[V(x, u) - V(x^+,u)] - \frac{\nu\beta}{2\gamma}\|\xl - \xu^+\|^2.
		\end{aligned}
		\end{align}
		
		Comparing the above equation with \eqref{eqQBoundSpec}, it now remains to show that 
		\begin{align}
			\label{tmp5}
			\alpha_T(1+p_T) = \Lambda_T\alpha_1p_1 + 1-\Lambda_T(1-\alpha_1).
		\end{align}
By \eqref{eqLambdaConv}, the last relation in \eqref{eqAGSpar}, and the fact that $\Lambda_1=1$, we have
		\begin{align}
			\frac{\alpha\tn p\tn}{\Lambda\tn} = \frac{\alpha_tp_t}{\Lambda_t} + \frac{\alpha_t}{\Lambda_t} = \ldots = \frac{\alpha_1p_1}{\Lambda_1} + \sumt[t]\frac{\alpha_i}{\Lambda_i} = \alpha_1p_1 + \frac{1-\Lambda_t(1-\alpha_1)}{\Lambda_t}.
		\end{align}
		Using the second relation in \eqref{eqAGSpar} to the above equation, we have
		\begin{align}
			\frac{\alpha_t(1+p_t)}{\Lambda_t} = \alpha_1p_1 + \frac{1-\Lambda_t(1-\alpha_1)}{\Lambda_t},
		\end{align}
		which implies
		$
			\alpha_t(1+p_t) = \Lambda_t\alpha_1p_1 + 1-\Lambda_t(1-\alpha_1)
		$
		for any $ 1 \le t \le T$.
	\qed\end{proof}

\vgap

With the help of the above proposition and Lemma \ref{lemPsiQ}, we are now ready to 
establish the convergence of the AGS method. Note that the following sequence will the used in the analysis of
the AGS method:
	\begin{align}
	\label{eqGamma}
	\Gamma_k = \begin{cases}
	1 & k=1
	\\
	(1-\gamma_k)\Gamma\kp & k>1.
	\end{cases}
	\end{align}

\vgap

\begin{theorem}	
	\label{thmGSconvergence}
	Suppose that \eqref{eqGScond} and \eqref{eqAGSpar} hold. If
	\begin{align}
	\label{eqcondbeta}
	\gamma_1=1\text{ and }\beta_k\ge \frac{L\gamma_k}{\nu},
	\end{align}
	then
	\begin{align}
	\label{eqGSconvergence}
	\begin{aligned}
	&\phi(\xu_k) - \phi(u)
	\\
	\le& \Gamma_k\sumt[k]\frac{\lambda_i\alpha_{T_i}(\beta_i(1+p_{T_i})+q_{T_i})}{\Gamma_i}(V(x\tp[i],u) - V(x_i, u)),
	\end{aligned}
	\end{align}
	where $\Gamma_k$ is defined in \eqref{eqGamma}.
\end{theorem}

	\begin{proof}
		It follows from Proposition \ref{proQBoundSpec} that for all $u\in X$,
		\begin{align}
		& Q_k(\xu_k, u) - (1-\gamma_k)Q_k(\xu\kp, u) 
		\\
		\le & \lambda_k\alpha_{T_k}(\beta_k(1+p_{T_k})+q_{T_k})(V(x\kp,u) - V(x_k, u)) - \frac{\nu\beta_k}{2\gamma_k}\|\xu_k - \xl_k\|^2.
		\end{align}
		Substituting the above bound to \eqref{eqPsiQ} in Lemma \ref{lemPsiQ}, and using \eqref{eqcondbeta}, we have
		\begin{align}
\begin{aligned}
		& \phi(\xu_k) - \phi(u) 
		\\
		\le & (1-\gamma_k)[\phi(\xu\kp) - \phi(u)] 
		\\
		&+ \lambda_k\alpha_{T_k}(\beta_k(1+p_{T_k})+q_{T_k})(V(x\kp,u) - V(x_k, u)),
\end{aligned}		\end{align}
		which, in view of Lemma \ref{lemSum} (with $c_k=\gamma_k$, $C_k = \Gamma_k$, and $\delta_k = \phi(\xu_k) - \phi(u)$), then implies that 
		\begin{align}
		& \phi(\xu_k) - \phi(u) 
		\\
		\le & \Gamma_k\left[\frac{1-\gamma_1}{\Gamma_1}(\phi(\xu_0) - \phi(u)) \right.
		\\
		&+\left. \sumt[k]\frac{\lambda_i\alpha_{T_i}(\beta_i(1+p_{T_i})+q_{T_i})}{\Gamma_i}(V(x\tp[i],u) - V(x_i, u))\right]
		\\
		= & \Gamma_k\sumt[k]\frac{\lambda_i\alpha_{T_i}(\beta_i(1+p_{T_i})+q_{T_i})}{\Gamma_i}(V(x\tp[i],u) - V(x_i, u)),
		\end{align}
		where the last equality follows from the fact that $\gamma_1=1$ in \eqref{eqcondbeta}. 
	\qed\end{proof}

\vgap

There are many possible selections of parameters that satisfy the assumptions of the above theorem. 
In the following corollaries we describe two different ways to specify the parameters of Algorithm \ref{algAGS} 
that lead to the optimal complexity bounds in terms of the number of gradient evaluations of
$\nabla f$ and $\nabla h$.

\vgap

\begin{corollary}
	\label{corAGSeasy}
	Consider problem \eqref{eqSmoothProblem} with the Lipschitz constants in \eqref{eqL} and \eqref{eqM} satisfing $M\ge L$.
	Suppose that the parameters of Algorithm \ref{algAGS} are set to
	\begin{align}
	\label{eqCond}
	\begin{aligned}
	& \gamma_k = \frac{2}{k+1},\ 
	T_k\equiv T:=\left\lceil\sqrt{\frac{M}{L}}\right\rceil,
	\\
	& \lambda_k = \begin{cases}
	1 & k=1,\\\dfrac{\gamma_k(T+1)(T+2)}{T(T+3)} & k>1,
	\end{cases}
	 \text{ and } 
	\beta_k = \frac{3L\gamma_k}{\nu k\lambda_k}.
	\end{aligned}	
	\end{align}
	Also assume that
	the parameters in the first call to the ProxAG procedure ($k=1$) are set to
	\begin{align}
	\label{eqCondAGS1}
	\alpha_t = \frac{2}{t+1},\ \ p_t = \frac{t-1}{2}, \text{ and }q_t = \frac{6M}{\nu t},
	\end{align}
	and the parameters in the remaining calls to the ProxAG procedure ($k>1$) are set to
	\begin{align}
	\label{eqCondAGS2}
	\alpha_t = \frac{2}{t+2},\ \ p_t = \frac{t}{2}, \text{ and }q_t = \frac{6M}{\nu k(t+1)}.
	\end{align}
	Then the numbers of gradient evaluations of $\nabla f$ and $\nabla h$ performed by the AGS method to 
	compute an $\epsilon$-solution of \eqref{eqSmoothProblem} can be bounded by 
	\begin{align}
	\label{eqNfeasy}
	N_{f}:=\sqrt{\frac{30LV(x_0, x^*)}{\nu\epsilon}}
	\end{align}
	and
	\begin{align}
	\label{eqNheasy}
	N_{h} := \sqrt{\frac{30MV(x_0, x^*)}{\nu\epsilon}} + \sqrt{\frac{30LV(x_0, x^*)}{\nu\epsilon}}
	\end{align}
	respectively, where $x^*$ is a solution to \eqref{eqSmoothProblem}.
\end{corollary}

	\begin{proof}
		Let us start with verification of \eqref{eqGScond}, \eqref{eqAGSpar}, and \eqref{eqcondbeta} for the purpose of applying Theorem \ref{thmGSconvergence}. We will consider the first call to the ProxAG procedure ($k=1$) and the remaining calls ($k>1$) separately.
		
		When $k=1$, by \eqref{eqCond} we have $\lambda_1=\gamma_1=1$, and $\beta_1=3L/\nu$, hence \eqref{eqcondbeta} holds immediately. By \eqref{eqCondAGS1} we can observe that $\Lambda_t=2/(t(t+1))$ satisfies \eqref{eqLambda}, and that
		\begin{align}
			\frac{\alpha_tq_t}{\Lambda_t} \equiv \frac{6M}{\nu},\text{ and }\frac{\alpha_t(1+p_t)}{\Lambda_t} = \frac{t(t+1)}{2}=\frac{\alpha\tn p\tn}{\Lambda\tn},
		\end{align}
		hence \eqref{eqAGSpar} holds. In addition, by \eqref{eqCond} and \eqref{eqCondAGS1} we have $\lambda=\gamma=1$ and $\alpha_1=1$ in \eqref{eqGScond}, and that
		\begin{align}
			\beta p_t + q_t \ge q_t = \frac{6M}{\nu t} > \frac{2M}{\nu(t+1)} = \frac{\lambda M \alpha_t}{\nu}.
		\end{align}
		Therefore \eqref{eqGScond} also holds.
		
		For the case when $k>1$, we can observe from \eqref{eqCondAGS2} that $\Lambda_t=6/(t+1)(t+2)$ satisfies \eqref{eqLambda}, ${\alpha_tq_t}/{\Lambda_t}\equiv{2M}/{(\nu k)}$, and that
		\begin{align}
		\frac{\alpha_t(1+p_t)}{\Lambda_t} = \frac{(t+1)(t+2)}{6}=\frac{\alpha\tn p\tn}{\Lambda\tn}.
		\end{align}
		Therefore \eqref{eqAGSpar} holds. Also, from \eqref{eqCond} and noting that $k, T\ge 1$, we have
		\begin{align}
		\label{eqtmp1}
		\frac{3}{k} > \frac{3\gamma_k}{2} = \frac{3\lambda_k}{2}\left(1 - \frac{2}{(T+1)(T+2)}\right) \ge \frac{3\lambda_k}{2}\left(1-\frac{2}{2\cdot 3}\right) = \lambda_k.
		\end{align}
		Applying the above relation to the definition of $\beta_k$ in \eqref{eqCond} we have \eqref{eqcondbeta}. It now suffices to verify \eqref{eqGScond} in order to apply Theorem \ref{thmGSconvergence}. Applying \eqref{eqCond}, \eqref{eqCondAGS2}, \eqref{eqtmp1}, and noting that $k\ge 2$ and that $\Lambda_T=6/(T+1)(T+2)$ with $T\ge 1$, we can verify in \eqref{eqGScond} that
		\begin{align}
		& \lambda =  \frac{\gamma(T+1)(T+2)}{T(T+3)} = \frac{2}{k+1}\left(1+\frac{2}{T(T+3)}\right)\le \frac{2}{3}\left(1+\frac{2}{1\cdot 4}\right)=1,
		\\
		&\Lambda_T(1-\alpha_1) = \frac{2}{(T+1)(T+2)} = 1 - \frac{T(T+3)}{(T+1)(T+2)} = 1- \frac{\gamma}{\lambda},
		\\
		&\beta p_t + q_t> q_t = \frac{2M}{\nu (t+1)}\cdot\frac{3}{k}>\frac{2\lambda M}{\nu (t+1)}\ge \frac{\lambda M\alpha_t}{\nu}.
		\end{align}
		Therefore, the conditions in \eqref{eqGScond} are satisfied. 
		
		We are now ready to apply Theorem \ref{thmGSconvergence}. In particular, noting that $\alpha_t(1+p_t)\equiv 1$ from \eqref{eqCondAGS1} and \eqref{eqCondAGS2}, we obtain from \eqref{eqGSconvergence} (with $u=x^*$) that
		\begin{align}
		\label{eqxiseq}
		\phi(\xu_k) - \phi^*\le \Gamma_k\sumt[k]\xi_i(V(x\tp[i],x^*) - V(x_i, x^*)),
		\end{align}
		where
		\begin{align}
		\label{eqxi}
		\xi_i:=\frac{\lambda_i(\beta_i+\alpha_{T_i}q_{T_i})}{\Gamma_i},
		\end{align}
		
		Substituting \eqref{eqCond} and \eqref{eqCondAGS1} to \eqref{eqxi}, and noting that $\Gamma_i=2/(i(i+1))$ by \eqref{eqGamma}, we have
		\begin{align}
		\xi_1 = & \beta_1 + \alpha_{T} q_{T} =  \frac{3L}{\nu} + \frac{12 M}{\nu T(T+1)},\text{ and }
		\\
		\xi_i = & \frac{\lambda_i\beta_i}{\Gamma_i} + \frac{\lambda_i\alpha_{T_i} q_{T_i}}{\Gamma_i} = \frac{3L\gamma_i}{\nu i\Gamma_i} + \frac{\gamma_i}{\Gamma_i}\frac{(T_i+1)(T_i+2)}{T_i(T_i+3)}\frac{2}{T_i+2}\frac{6M}{\nu i(T_i+1)}
		\\
		\equiv & \frac{3L}{\nu} + \frac{12M}{\nu T(T+3)},\forall i>1.
		\end{align}
		Applying the above two results regarding $\xi_i$ to \eqref{eqxiseq}, and noting that $\xi_1>\xi_2$, we have
		\begin{align}
		& \phi(\xu_k) - \phi^* 
		\\
		\le & \Gamma_k\left[\xi_1(V(x_0,x^*) - V(x_1, x^*)) + \sum_{i=2}^k\xi_i(V(x_{i-1},x^*) - V(x_i, x^*)) \right]
		\\
		= & \Gamma_k\left[\xi_1(V(x_0,x^*) - V(x_1, x^*)) + \xi_2(V(x_1,x^*) - V(x_k, x^*)) \right]
		\\
		\le & \Gamma_k\xi_1V(x_0,x^*) 
		\\
		= & \frac{2}{k(k+1)}\left(\frac{3L}{\nu} + \frac{12M}{\nu T(T+1)}\right)V(x_0, x^*) 
		\\
		\le & \frac{30L}{\nu k(k+1)}V(x_0, x^*),
		\end{align}
		where the last inequality is due to the fact that $T\ge \sqrt{M/L}$.
		
		From the above inequality, the number of calls to the ProxAG procedure for computing an $\epsilon$-solution of \eqref{eqSmoothProblem} is bounded by $N_{f}$ in \eqref{eqNfeasy}.
		This is also the bound for the number of gradient evaluations of $\nabla f$. Moreover, the number of gradient evaluations of $\nabla h$ is bounded by
		\begin{align}
		TN_{f} \le \left(\sqrt{\frac{M}{L}}+1\right)N_{ f} = \sqrt{\frac{30MV(x_0, x^*)}{\nu\epsilon}}+\sqrt{\frac{30LV(x_0, x^*)}{\nu\epsilon}} = N_{ h}.
		\end{align}
	\qed\end{proof}

\vgap

In the above corollary, the constant factors in \eqref{eqNfeasy} and \eqref{eqNheasy} are both given by $\sqrt{30}$. 
In the following corollary, we provide a slightly different set of parameters for Algorithm \ref{algAGS} 
that results in a tighter constant factor for \eqref{eqNfeasy}.

\vgap

\begin{corollary}
	\label{corAGS}
	Consider problem \eqref{eqSmoothProblem} with the Lipschitz constants in \eqref{eqL} and \eqref{eqM} satisfing $M\ge L$.
	Suppose that the parameters in the first call to the ProxAG procedure ($k=1$) are set to
	\begin{align}
	\label{eqcondAGS1inner}
	\alpha_t=\frac{2}{t+1},\ p_t = \frac{t-1}{2}, \text{ and } q_t = \frac{7LT(T+1)}{4\nu t},
	\end{align}
	and that the parameters in the $k$-th call ($k>1$) are set to
	\begin{align}
	\label{eqcondAGSkinner}
	p_t\equiv p:=\sqrt{\frac{M}{L}},\ \alpha_t\equiv\alpha:=\frac{1}{p+1},\text{ and }q_t\equiv 0.
	\end{align}
	If the other parameters in Algorithm \ref{algAGS} satisfy
	\begin{align}
	\label{eqcondAGSouter}
\begin{aligned}
	& \gamma_k =\dfrac{2}{k+1}, 
	T_k:=\begin{dcases}
	\left\lceil\sqrt{\dfrac{8M}{7L}}\right\rceil, & k=1
	\\
	\left\lceil\dfrac{\ln(3)}{-\ln(1-\alpha)}\right\rceil, & k>1,
	\end{dcases}
	\\
	&\lambda_k:=\begin{dcases}
	1, & k=1
	\\
	\dfrac{\gamma_k}{1 - (1-\alpha)^{T_k}}, & k>1,
	\end{dcases}
 	\text{ and } \beta_k:=\begin{dcases}
	\dfrac{L}{\nu}, & k=1
	\\
	\dfrac{9L\gamma_k}{2\nu k\lambda_k}, & k>1,
	\end{dcases}
\end{aligned}	\end{align}
	where $\alpha$ is defined in \eqref{eqcondAGSkinner}, 
	then the numbers of gradient evaluations of $\nabla f$ and $\nabla h$ performed
	by the AGS method to find an $\varepsilon$-solution to problem \eqref{eqSmoothProblem} can be bounded by 
	\begin{align}
	\label{eqNf}
	N_{f}:=3\sqrt{\frac{LV(x_0, x^*)}{\nu\epsilon}}
	\end{align}
	and
	\begin{align}
	\label{eqNh}
	\begin{aligned}
	N_{h} := &(1+\ln 3)N_{f}\left(\sqrt{\frac{M}{L}}+1\right) 
	\\
	\le &7\left(\sqrt{\frac{MV(x_0, x^*)}{\nu\epsilon}} + \sqrt{\frac{LV(x_0, x^*)}{\nu\epsilon}}\right),
	\end{aligned}
	\end{align}
	respectively.
\end{corollary}

	\begin{proof}
		Let us verify \eqref{eqGScond}, \eqref{eqcondbeta}, and \eqref{eqAGSpar} first, so that we could apply Theorem \ref{thmGSconvergence}. We consider the case when $k=1$ first. By the definition of $\gamma_k$ and $\beta_k$ in \eqref{eqcondAGSouter}, it is clear that \eqref{eqcondbeta} is satisfied when $k=1$. 
		Also, by \eqref{eqcondAGS1inner} we have that $\Lambda_t=2/(t(t+1))$ in \eqref{eqLambda}, 
		\begin{align}
			\frac{\alpha_tq_t}{\Lambda_t}\equiv\frac{7LT_1(T_1+1)}{4\nu},\text{ and }\frac{\alpha_t(1+p_t)}{\Lambda_t}=\frac{t(t+1)}{2}=\frac{\alpha\tn p\tn}{\Lambda\tn},
		\end{align}
		hence \eqref{eqAGSpar} also holds. Moreover, by \eqref{eqcondAGS1inner} and \eqref{eqcondAGSouter}, we can verify in \eqref{eqGScond} that
		\begin{align}
			\lambda = \gamma = 1, \Lambda_{T_1}(1-\alpha_1) = 0 = 1-\frac{\gamma}{\lambda},
		\end{align}
		and 
		\begin{align}
			\beta p_t + q_t\ge q_t > \frac{7LT^2}{4\nu t} = \frac{8M}{4\nu t}> \frac{M\alpha_t}{\nu}.
		\end{align}
		Therefore the relations in \eqref{eqGScond} are all satisfied.
		
		Now we consider the case when $k>1$. By \eqref{eqLambda} and \eqref{eqcondAGSkinner}, we observe that $\Lambda_t = (1-\alpha)^{t-1}$ for all $t\ge 1$. Moreover, from the definition of $T_k$ in \eqref{eqcondAGSouter}, we can also observe that
		\begin{align}
		(1-\alpha)^{T_k}\le \frac{1}{3}.
		\end{align}
		Four relations can be derived based on the aforementioned two observations, \eqref{eqcondAGSkinner}, and \eqref{eqcondAGSouter}. 
		First,
		\begin{align}
		\frac{\alpha_t q_t}{\Lambda_t}\equiv 0,\ 
		\frac{\alpha_t(1+p_t)}{\Lambda_t} = \frac{1}{(1-\alpha)^{t-1}} = \frac{\alpha\tn p\tn}{\Lambda\tn},
		\end{align}
		which verifies \eqref{eqAGSpar}. 
		Second,
		\begin{align}
		\beta_k = \frac{9L(1-(1-\alpha)^{T_k})}{2\nu k}\ge \frac{3L}{\nu k}>\frac{L\gamma_k}{\nu},
		\end{align}
		which leads to \eqref{eqcondbeta}. 
		Third, noting that $k\ge 2$, we have
		\begin{align}
		\frac{\gamma_k}{1-\Lambda_{T_k}(1-\alpha)} = \lambda_k =\frac{\gamma_k}{1-(1-\alpha)^{T_k}}\le \frac{3\gamma_k}{2} = \frac{3}{k+1}\le 1.
		\end{align}
		Fourth, 
		\begin{align}
		\frac{\nu\beta_kp}{\lambda_k M\alpha} 
		= & \frac{9L\gamma_kp(p+1)}{2k\lambda_k^2M} = \frac{9Lp(p+1)\left(1-(1-\alpha)^{T_k}\right)^2}{2k\gamma_kM} 
		\\
		= & \frac{9(k+1)}{4k}\cdot\left(\frac{Lp(p+1)}{M}\right)\cdot\left(1-(1-\alpha)^{T_k}\right)^2
		\\
		> & \frac{9}{4} \cdot 1 \cdot \frac{4}{9} = 1.
		\end{align}
		The last two relations imply that \eqref{eqGScond} holds. 
		
		Summarizing the above discussions regarding both the cases $k=1$ and $k>1$, applying Theorem \ref{thmGSconvergence}, and noting that $\alpha_t(1+p_t)\equiv 1$, we have
		\begin{align}
		\label{eqpsixi}
		\phi(\xu_k) - \phi(u)\le \Gamma_k\sumt[k]\xi_i(V(x\tp[i],u)-V(x_i,u)),\ \forall u\in X,
		\end{align}
		where
		\begin{align}
		\xi_i:=\frac{\lambda_i(\beta_i+\alpha_{T_i}q_{T_i})}{\Gamma_i}.
		\end{align}
		It should be observed from the definition of $\gamma_k$ in \eqref{eqcondAGSouter} that $\Gamma_i:=2/(i(i+1))$ satisfies \eqref{eqGamma}. Using this observation, applying \eqref{eqcondAGS1inner}, \eqref{eqcondAGSkinner}, and \eqref{eqcondAGSouter} to the above equation we have
		\begin{align}
		\xi_1 = \beta_1 + \alpha_{T_1}q_{T_1} = \frac{L}{\nu} + \frac{7L}{2\nu} =  \frac{9L}{2\nu}
		\end{align}
		and
		\begin{align}
		\xi_i=\frac{\lambda_i\beta_i}{\Gamma_i}\equiv\frac{9L}{2\nu}, \ \forall i>1.
		\end{align}
		Therefore, \eqref{eqpsixi} becomes
		\begin{align}
		\label{eqk2}
		\begin{aligned}
		\phi(\xu_k) - \phi(u) \le& \frac{9L}{\nu k(k+1)}(V(x_0,u) - V(x_k, u)) 
		\\
		\le &\frac{9L}{\nu k(k+1)}V(x_0, u).
		\end{aligned}
		\end{align}
		
		Setting $u=x^*$ in the above inequality, we observe that the number of calls to the ProxAG procedure for computing an $\epsilon$-solution of \eqref{eqSmoothProblem} is bounded by $N_{f}$ in \eqref{eqNf}.
		This is also the bound for the number of gradient evaluations of $\nabla f$. Moreover, by \eqref{eqcondAGSkinner},  \eqref{eqcondAGSouter}, and \eqref{eqNf} we conclude that the number of gradient evaluations of $\nabla h$ is bounded by
		\begin{align}
		\sum_{k=1}^{N_{f}}T_k =& T_1 + \sum_{k=2}^{N_{f}}T_k
		\le  \left(\sqrt{\frac{8M}{7L}}+1\right) + (N_{f}-1)\left(\frac{\ln{3}}{-\ln(1-\alpha)} +1\right)
		\\
		\le & 
		\left(\sqrt{\frac{8M}{7L}}+1\right) + (N_{f}-1)\left(\frac{\ln 3}{\alpha} + 1\right)
		\\
		= & \left(\sqrt{\frac{8M}{7L}}+1\right) + (N_{f}-1)\left(\left(\sqrt{\frac{M}{L}}+1\right)\ln 3+1\right)
		\\
		< & (1+\ln 3)N_{f}\left(\sqrt{\frac{M}{L}}+1\right)
		\\
		< & 7\left(\sqrt{\frac{MV(x_0, x^*)}{\nu\epsilon}} + \sqrt{\frac{LV(x_0, x^*)}{\nu\epsilon}}\right).
		\end{align}
		Here the second inequity is from the property of logarithm functions that $-\ln(1-\alpha)\ge \alpha$ for $\alpha\in[0,1)$.
	\qed\end{proof}

\vgap

Since $M\ge L$ in \eqref{eqL} and \eqref{eqM}, the results obtained in Corollaries \ref{corAGSeasy} and \ref{corAGS} indicate 
that the number of gradient evaluations of $\nabla f$ and $\nabla h$ that Algorithm \ref{algAGS} requires for computing 
an $\varepsilon$-solution of \eqref{eqSmoothProblem} can be bounded by $\cO(\sqrt{L/\varepsilon})$ and $\cO(\sqrt{M/\varepsilon})$, respectively. 
Such a result is particularly useful when $M$ is significantly larger, e.g., $M=\cO(L/\varepsilon)$, since the number of 
gradient evaluations of $\nabla f$ would not be affected at all by the large Lipschitz constant of the whole problem. 
It is interesting to compare the above result with the best known so-far complexity bound under the traditional black-box oracle assumption. 
If we treat problem \eqref{eqSmoothProblem} as a general smooth convex optimization and study its oracle complexity, i.e.,
under the assumption that there exists an \emph{oracle} that outputs $\nabla \phi(x)$ for any test point $x$ (and $\nabla \phi(x)$ only),
it has been shown that the number of calls to the oracle cannot be smaller than $\cO(\sqrt{(L+M)/\varepsilon})$ for computing an $\varepsilon$-solution \cite{nemirovski1983problem,nesterov2004introductory}. Under such ``single oracle'' assumption, the complexity bounds 
in terms of gradient evaluations of $\nabla f$ and $\nabla h$ are intertwined, and a larger Lipschitz constant $M$ 
will result in more gradient evaluations of $\nabla f$, even though there is no explicit relationship between $\nabla f$ and $M$.
However, the results in Corollaries \ref{corAGSeasy} and \ref{corAGS} suggest that we can study the oracle complexity of 
problem \eqref{eqSmoothProblem} based on the assumption of \emph{two separate oracles}: one oracle $\cO_{f}$ to compute $\nabla f$ for any 
test point $x$, and the other one $\cO_{h}$ to compute $\nabla h(y)$ for any test point $y$. In particular, these two oracles do not 
have to be called at the same time, and hence it is possible to obtain separate complexity bounds $\cO(\sqrt{L/\varepsilon})$ 
and $\cO(\sqrt{M/\varepsilon})$ on the number of calls to $\cO_{f}$ and $\cO_{h}$, respectively.


\vgap

We now consider a special case of \eqref{eqSmoothProblem} where $f$ is strongly convex. More specifically, we assume that there exists $\mu>0$ such that
\begin{align}
\label{eqmu}
\frac{\mu}{2}\|x - u\|^2 \le f(x) - l_f(u,x) \le \frac{L}{2}\|x - u\|^2,\ \forall x,u\in X.
\end{align}
Under the above assumption, we develop a multi-stage AGS algorithm that can skip computation of $\nabla f$ from time to time, and compute an $\varepsilon$-solution of \eqref{eqSmoothProblem} with
\begin{align}
\label{eqStronglyConvexBound}
\cO\left(\sqrt{\frac{L}{\mu}}\log\frac{1}{\varepsilon}\right)
\end{align}
gradient evaluations of $\nabla f$ (see Alagorithm~\ref{algAGSR}). 
It should be noted that, under the traditional black-box setting~\cite{nemirovski1983problem,nesterov2004introductory} where 
one could only access $\nabla \psi(x)$ for each inquiry $x$, the number of evaluations of $\nabla \psi(x)$ required to compute an $\varepsilon$-solution is bounded by
\begin{align}
	\label{eqPrevStrong}
	\cO\left(\sqrt{\frac{L+M}{\mu}}\log\frac{1}{\varepsilon}\right).
\end{align}

\begin{algorithm}
	\caption{\label{algAGSR}The multi-stage accelerated gradient sliding (M-AGS) algorithm}
	\begin{algorithmic}
		\State Choose $v_0\in X$, accuracy $\varepsilon$, iteration limit $N_0$, and initial estimate $\Delta_0$ such that $\phi(v_0)-\phi^*\le \Delta_0$. 
		\For {$s=1,\ldots,S$}
		\State Run the AGS algorithm with $x_0=v\tp[s]$, $N=N_0$, and parameters in Corollary \ref{corAGS}, and let $v_s=\xu_{N}$.
		\EndFor
		\State Output $v_S$.
	\end{algorithmic}
\end{algorithm}

\vgap

Theorem \ref{thmAGSR} below describes the main convergence properties of the M-AGS algorithm.

\vgap

\begin{theorem}
	\label{thmAGSR}
	Suppose that $M\ge L$ in \eqref{eqM} and \eqref{eqmu}, and that the prox-function $V(\cdot,\cdot)$ grows quadratically (i.e., \eqref{eqVquad} holds). 
	If the parameters in Algorithm \ref{algAGSR} are set to
	\begin{align}
		\label{eqAGSRpar}
		N_0=3\sqrt{\frac{2L}{\nu\mu}}\text{ and }S=\log_2\max\left\{\frac{\Delta_0}{\varepsilon}, 1\right\},
	\end{align}
	then its output $v_S$ must be an $\varepsilon$-solution of \eqref{eqSPP}. Moreover, the total number of 
	gradient evaluations of $\nabla f$ and $\nabla h$ performed by Algorithm \ref{algAGSR} can be bounded by
	\begin{align}
		\label{eqnf}
		N_{f}:=3\sqrt{\frac{2L}{\nu\mu}}\log_2\max\left\{\frac{\Delta_0}{\varepsilon}, 1\right\}
	\end{align}
	and
	\begin{align}
	\label{eqnh}
	\begin{aligned}
	N_{h}:=&(1+\ln 3)N_{f}\left(\sqrt{\frac{M}{L}}+1\right) 
	\\
	< &9\left(\sqrt{\frac{L}{\nu\mu}}+\sqrt{\frac{M}{\nu\mu}}\right)\log_2\max\left\{\frac{\Delta_0}{\varepsilon}, 1\right\},
	\end{aligned}
	\end{align}
	respectively.
	\end{theorem}

\begin{proof}
		With input $x_0=v\tp[s]$ and $N=N_0$, we conclude from \eqref{eqk2} in the proof of Corollary \ref{corAGS} (with $u=x^*$ a solution to \eqref{eqSmoothProblem}) that
		\begin{align}
			\phi(\xu_N)-\phi^*\le \frac{9L}{\nu N_0(N_0+1)}V(x_0,x^*) \le \frac{\mu}{2}V(x_0,x^*),
		\end{align}
		where the last inequality follows from \eqref{eqAGSRpar}. Using
		the facts that the input of the AGS algorithm is $x_0=v\tp[s]$ and that the output is set to $v_s=\xu_N$, and the relation \eqref{eqVquad}, we conclude
		\begin{align}
			\phi(v_s)-\phi^*\le \frac{\mu}{4}\|v\tp[s] - x^*\|^2\le \frac{1}{2}(\phi(v\tp[s])-\phi^*),
		\end{align}
		where the last inequality is due to the strong convexity of $\phi(\cdot)$. 
		It then follows from the above relation, the definition of $\Delta_0$ in Algorithm \ref{algAGSR}, and \eqref{eqAGSRpar} that
		\begin{align}
			\phi(v_S)-\phi^*\le \frac{1}{2^S}(\phi(v_0)-\phi^*) \le \frac{\Delta_0}{2^S} \le \varepsilon.
		\end{align}
		Comparing Algorithms \ref{algAGS} and \ref{algAGSR}, we can observe that the total number of gradient evaluations of $\nabla f$ in 
		Algorithm~\ref{algAGSR} is bounded by $N_0S$, and hence we have \eqref{eqnf}. Moreover, comparing \eqref{eqNf} and \eqref{eqNh} 
		in Corollary \ref{corAGS}, we conclude \eqref{eqnh}.
	\qed\end{proof}

\vgap

In view of Theorem~\ref{thmAGSR}, the total number of gradient evaluations of $\nabla h$ required by the M-AGS algorithm to 
compute an $\epsilon$-solution of \eqref{eqSmoothProblem} is the same as the traditional result \eqref{eqPrevStrong}. 
However, by skipping the  gradient evaluations of $\nabla f$ from time to time in the M-AGS algorithm,
the total number of gradient evaluations of $\nabla f$ is improved from \eqref{eqPrevStrong} to \eqref{eqStronglyConvexBound}. 
Such an improvement becomes more significant as the ratio $M/L$ increases.

\section{Application to composite bilinear saddle point problems}
\label{secSPP}

Our goal in this section is to show the advantages of the AGS method when applied to our motivating problem, i.e.,
the composite bilinear saddle point problem in \eqref{eqSPP}. In particular, we show in Section \ref{secAGSSPP} that the AGS algorithm can be used to solve \eqref{eqSPP} by incorporating the smoothing technique in \cite{nesterov2005smooth} and derive new complexity bounds 
in terms of the number of gradient computations of $\nabla f$ and operator evaluations of $K$ and $K^T$. 
Moreover, we demonstrate in Section~\ref{secStrongSPP} that even more significant saving on gradient computation of $\nabla f$ can be obtained
when $f$ is strongly convex in \eqref{eqSPP} by incorporating the multi-stage AGS method. 

\subsection{Saddle point problems}
\label{secAGSSPP}

Our goal in this section is to extend the AGS algorithm from composite smooth optimization to nonsmooth optimization. By incorporating the smoothing technique in \cite{nesterov2005smooth}, we can apply AGS to solve the composite saddle point problem \eqref{eqSPP}. Throughout this section, we assume that
the dual feasible set $Y$ in \eqref{eqSPP} is bounded, i.e., there exists $y_0\in Y$ such that
\begin{align}
\label{eqOmegaY}
	\Omega:=\max_{v\in Y}W(y_0,v)
\end{align}
is finite, where $W(\cdot,\cdot)$ is the prox-function associated with $Y$ with modulus $\omega$.

Let $\psi_\rho$ be the smooth approximation of $\psi$ defined in \eqref{eqSmoothApproxProblem}.
It can be easily shown (see \cite{nesterov2005smooth}) that
\begin{align}
	\label{eqpsipsirho}
	\psi_\rho(x)\le \psi(x)\le\psi_\rho(x)+\rho\Omega ,\ \forall x\in X.
\end{align}
Therefore, if
$\rho =\varepsilon/(2\Omega)$, 
then an $(\varepsilon/2)$-solution to \eqref{eqSmoothApproxProblem} is also an $\varepsilon$-solution to \eqref{eqSPP}. Moreover, 
it follows from Theorem 1 in \cite{nesterov2005smooth} that problem \eqref{eqSmoothApproxProblem} is 
given in the form of \eqref{eqSmoothProblem} (with $h(x)=h_{\rho}(x)$) and satisfies \eqref{eqM} with $M=\|K\|^2/(\rho \omega)$.
Using these observations, we are ready to summarize the convergence properties of the AGS algorithm for solving problem \eqref{eqSPP}.

\vgap

\begin{proposition}
	\label{corAGSS}
	Let $\varepsilon > 0$ be given and assume that $2\|K\|^2\Omega >\varepsilon\omega L$. If we apply the AGS method in 
	Algorithm \ref{algAGS} to problem \eqref{eqSmoothApproxProblem} (with $h=h_{\rho}$ and $\rho=\varepsilon/(2\Omega)$), in which the parameters 
	are set to \eqref{eqcondAGS1inner}--\eqref{eqcondAGSouter} with $M=\|K\|^2/(\rho \omega)$, then the total number 
	of gradient evaluations of $\nabla f$ and linear operator evaluations of $K$ (and $K^T$) in order to find 
	an $\varepsilon$-solution of \eqref{eqSPP} can be bounded by
	\begin{align}
	\label{eqNfSPP}
	N_{f}:=3\left(\sqrt{\frac{2LV(x_0,x^*)}{\nu\varepsilon}}\right)
	\end{align}
	and
	\begin{align}
	\label{eqKBoundAGS}
	N_K:=14\left(\sqrt{\frac{2LV(x_0,x^*)}{\nu\varepsilon}} + \frac{2\|K\|\sqrt{V(x_0,x^*)\Omega }}{\sqrt{\nu\omega}\varepsilon}\right),
	\end{align}
	respectively.	
\end{proposition}

	\begin{proof}
		By \eqref{eqpsipsirho} we have $\psi_\rho^*\le \psi^*$ and $\psi(x)\le\psi_\rho(x)+\rho\Omega $ for all $x\in X$, and hence
		\begin{align}
			\psi(x) - \psi^*\le \psi_{\rho}(x)-\psi^*_{\rho} + \rho\Omega ,\ \forall x\in X.
		\end{align}
		Using the above relation and the fact that $\rho=\varepsilon/(2\Omega)$ we conclude that if $\psi_\rho(x) - \psi_\rho^*\le \varepsilon/2$, then $x$ is an $\varepsilon$-solution to \eqref{eqSPP}. 
		To finish the proof, it suffices to consider the complexity of AGS for computing an $\varepsilon/2$-solution of \eqref{eqSmoothApproxProblem}. 
		By Corollary \ref{corAGS}, the total number of gradient evaluations of $\nabla f$ is bounded by \eqref{eqNfSPP}. 
		By Theorem 1 in \cite{nesterov2005smooth}, the evaluation of $\nabla h_{\rho}$ is equivalent to $2$ evaluations of linear operators: 
		one computation of form $Kx$ for computing the maximizer $y^*(x)$ for problem \eqref{eqhrho}, and one computation of form $K^Ty^*(x)$ for 
		computing $\nabla h_{\rho}(x)$. Using this observation, and substituting $M=\|K\|^2/(\rho \omega)$ to \eqref{eqNh}, we conclude \eqref{eqKBoundAGS}.
	\qed\end{proof}
\vgap 

According to Proposition~\ref{corAGSS}, the total number of gradient evaluations of $\nabla f$ and linear operator evaluations of both $K$ and $K^T$ are bounded by
\begin{align}
	\label{eqOptSmoothXY} 
	\cO\left(\sqrt{\frac{L}{\varepsilon}}\right)
\end{align}
and
\begin{align}
\label{eqOptSPPXY} 
\cO\left(\sqrt{\frac{L}{\varepsilon}}+\frac{\|K\|}{\varepsilon}\right)
\end{align}
respectively, for computing an $\varepsilon$-solution of the saddle point problem \eqref{eqSPP}. Therefore, if $L\le \cO(\|K\|^2/\varepsilon)$, then 
the number of gradient evaluations of $\nabla f$ will not be affected by the dominating term $\cO(\|K\|/\varepsilon)$. 
This result significantly improves the best known so-far complexity results for solving the bilinear saddle point problem~\eqref{eqSPP}
in \cite{nesterov2005smooth} and \cite{lan2015gradient}. 
Specifically, it improves the complexity regarding number of gradient computations of $\nabla f$ from $\cO(1/\varepsilon)$ in \cite{nesterov2005smooth}
to $\cO(1/\sqrt{\varepsilon})$, and also improves the complexity regarding operator evaluations involving $K$ from $\cO(1/\varepsilon^2)$ in \cite{lan2015gradient} 
to $\cO(1/\varepsilon)$.


\subsection{Strongly convex composite saddle point problems}
\label{secStrongSPP}
In this subsection, we still consider the SPP in \eqref{eqSPP}, but assume that $f$ is strongly convex (i.e., \eqref{eqmu} holds). 
In this case, it has been shown previously in the literature that  $\cO(\|K\|/\sqrt{\varepsilon})$ 
first-order iterations, each one of them involving the computation of $\nabla f$, and the evaluation of $K$ and $K^T$,
are needed in order to  
compute an $\varepsilon$-solution of \eqref{eqSPP} (e.g., \cite{nesterov2005excessive}). However, we demonstrate in this subsection 
that the complexity with respect to the gradient evaluation of $\nabla f$ can be significantly improved from $\cO(1/\sqrt{\varepsilon})$ to $\cO(\log(1/\varepsilon))$.

Such an improvement can be achieved by properly restarting the AGS method applied to to solve a series of smooth optimization problem of form \eqref{eqSmoothApproxProblem}, in which the smoothing parameter $\rho$ changes over time. The proposed multi-stage AGS algorithm with dynamic smoothing is stated in Algorithm \ref{algAGSRS}.
\begin{algorithm}
	\caption{\label{algAGSRS}The multi-stage AGS algorithm with dynamic smoothing}
	\begin{algorithmic}
		\State Choose $v_0\in X$, accuracy $\varepsilon$, smoothing parameter $\rho_0$, iteration limit $N_0$, and initial estimate $\Delta_0$ of \eqref{eqSPP} such that $\psi(v_0)-\psi^*\le \Delta_0$. 
		\For {$s=1,\ldots,S$}
		\State Run the AGS algorithm to problem \eqref{eqSmoothApproxProblem} with $\rho=2^{-s/2}\rho_0$ (where $h=h_{\rho}$ in AGS). In the AGS algorithm, set $x_0=v\tp[s]$, $N=N_0$, and parameters in Corollary \ref{corAGS}, and let $v_s=\xu_{N}$.
		\EndFor
		\State Output $v_S$.
	\end{algorithmic}
\end{algorithm}

Theorem~\ref{thmAGSRS} describes the main convergence properties of Algorithm \ref{algAGSRS}.

\vgap

\begin{theorem}
	\label{thmAGSRS}
	Let $\varepsilon > 0$ be given and 
	suppose that the Lipschitz constant $L$ in \eqref{eqmu} satisfies
	\begin{align}
		\Omega \|K\|^2\max\left\{\sqrt{\frac{15\Delta_0}{\varepsilon}}, 1\right\}\ge 2\omega\Delta_0L.
	\end{align}
	Also assume that the prox-function $V(\cdot,\cdot)$ grows quadratically (i.e., \eqref{eqVquad} holds). 
	If the parameters in Algorithm \ref{algAGSRS} are set to
	\begin{align}
	\label{eqAGSRSpar}
	 N_0=3\sqrt{\frac{2L}{\nu\mu}},\ S=\log_2\max\left\{\frac{15\Delta_0}{\varepsilon}, 1\right\},\text{ and } \rho_0=\frac{4\Delta_0}{\Omega 2^{S/2}},
	\end{align}
	then the output $v_S$ of this algorithm must be an $\varepsilon$-solution \eqref{eqSPP}. Moreover, the total number of 
	gradient evaluations of $\nabla f$ and operator evaluations involving $K$ and $K^T$ performed by Algorithm \ref{algAGSRS} can be bounded by
	\begin{align}
	\label{eqnfS}
	N_{f}:=3\sqrt{\frac{2L}{\nu\mu}}\log_2\max\left\{\frac{15\Delta_0}{\varepsilon}, 1\right\}
	\end{align}
	and
	\begin{align}
	\label{eqnhS}
	N_{K}:=	18\sqrt{\frac{L}{\nu\mu}}\log_2\max\left\{\frac{15\Delta_0}{\varepsilon}, 1\right\} + \frac{56\sqrt{\Omega }\|K\|}{\sqrt{\mu\Delta_0\nu\omega}}\cdot\max\left\{\sqrt{\frac{15\Delta_0}{\varepsilon}}, 1\right\},
	\end{align}
	respectively.	
	\end{theorem}

\begin{proof}
		Suppose that $x^*$ is an optimal solution to \eqref{eqSPP}. By \eqref{eqk2} in the proof of Corollary \ref{corAGS}, in the $s$-th stage of Algorithm \ref{algAGSRS} 
		(calling AGS with input $x_0=v_{s-1}$, output $v_{s}=\xu_N$, and iteration number $N=N_0$), we have
		\begin{align}
			& \psi_{\rho}(v_s)-\psi_{\rho}(x^*)= \psi_{\rho}(\xu_N)-\psi_{\rho}(x^*)
			\\
			\le & \frac{9L}{\nu N_0(N_0+1)}V(x_0, x^*)\le \frac{\mu}{2}V(x_0,x^*)\le \frac{\mu}{4}\|x_0-x^*\|^2=\frac{\mu}{4}\|v\tp[s]-x^*\|^2,
		\end{align}
		where the last two inequalities follow from \eqref{eqAGSRSpar} and \eqref{eqVquad}, respectively. Moreover, by \eqref{eqpsipsirho} we have $\psi(v_s)\le \psi_\rho(v_s)+\rho\Omega $ and $\psi^*=\psi(x^*)\ge \psi_\rho(x^*)$, hence
		\begin{align}
			\psi(v_s)-\psi^*\le \psi_\rho(v_s)-\psi_\rho(x^*)+\rho\Omega .
		\end{align}
		Combing the above two equations and using the strong convexity of $\psi(\cdot)$, we have
		\begin{align}
			& \psi(v_s)-\psi^*
			\le  \frac{\mu}{4}\|v\tp[s]-x^*\|^2+\rho\Omega  
			\\
			\le & \frac{1}{2}[\psi(v\tp[s])-\psi^*]+\rho\Omega  = \frac{1}{2}[\psi(v\tp[s])-\psi^*]+2^{-s/2}\rho_0\Omega,
		\end{align}
		where the last equality is due to the selection of $\rho$ in Algorithm \ref{algAGSRS}. Reformulating the above relation as
		\begin{align}
			2^s[\psi(v_s)-\psi^*]\le 2^{s-1}[\psi(v\tp[s])-\psi^*]+2^{s/2}\rho_0\Omega,
		\end{align}
		and summing the above inequalities from $s=1,\ldots,S$, we have
		\begin{align}
			& 2^S(\psi(v_S)-\psi^*)
			\\
			\le & \Delta_0 + \rho_0\Omega \sum_{s=1}^S 2^{s/2} 
			= \Delta_0 + \rho_0\Omega \frac{\sqrt{2}(2^{S/2}-1)}{\sqrt{2}-1}< \Delta_0+\frac{7}{2}\rho_0\Omega 2^{S/2}=15\Delta_0,
		\end{align}
		where the first inequality follows from the fact that $\psi(v_0)-\psi^*\le\Delta_0$ and
		the last equality is due to \eqref{eqAGSRSpar}. By \eqref{eqAGSRSpar} and the above result, we have $\psi(v_S)-\psi^*\le\varepsilon$. 
		Comparing the descriptions of Algorithms \ref{algAGS} and \ref{algAGSRS}, we can clearly see that 
		the total number of gradient evaluations of $\nabla f$ in Algorithm \ref{algAGSRS} is given $N_0S$, hence we have \eqref{eqnfS}.
		
		To complete the proof it suffices to estimate the total number of operator evaluations involving $K$ and $K^T$. 
		By Theorem 1 in \cite{nesterov2005smooth}, in the $s$-th stage of Algorithm \ref{algAGSRS}, the number of operator evaluations involving $K$
		is equivalent to twice the number of evaluations of $\nabla h_\rho$ in the AGS algorithm, which, in view of \eqref{eqNh} in Corollary \ref{corAGS}, is
		given by
		\begin{align}
			& 2(1+\ln 3)N\left(\sqrt{\frac{M}{L}}+1\right) 
			\\
			= & 2(1+\ln 3)N\left(\sqrt{\frac{\|K\|^2}{\rho\omega L}}+1\right) = 2(1+\ln 3)N_0\left(\sqrt{\frac{2^{s/2}\|K\|^2}{\rho_0\omega L}}+1\right),
		\end{align}
		where we used the relation $M=\|K\|^2/(\rho\omega)$ (see Section \ref{secAGSSPP}) in the first equality and  relations $\rho=2^{-s/2}\rho_0$ and $N=N_0$ from Algorithm \ref{algAGSRS} in the last equality. It then follows from the above result and \eqref{eqAGSRSpar}
		that the total number of operator evaluations involving $K$ in Algorithm \ref{algAGSRS} can be bounded by
		\begin{align}
			& \sum_{s=1}^S 2(1+\ln 3)N_0\left(\sqrt{\frac{2^{s/2}\|K\|^2}{\rho_0\omega L}}+1\right)
			\\
			= & 2(1+\ln 3)N_0S + \frac{2(1+\ln 3)N_0\|K\|}{\sqrt{\rho_0\omega L}}\sum_{s=1}^S 2^{s/4}
			\\
			= & 2(1+\ln 3)N_0S + \frac{3\sqrt{2}(1+\ln 3)\sqrt{\Omega }\|K\|2^{S/4}}{\sqrt{\mu\Delta_0\nu\omega}}\cdot\frac{2^{1/4}(2^{S/4}-1)}{2^{1/4}-1}
			\\
			< & 2(1+\ln 3)N_0S + \frac{56\sqrt{\Omega }\|K\|}{\sqrt{\mu\Delta_0\nu\omega}}\cdot 2^{S/2}
			\\
			< & 18\sqrt{\frac{L}{\nu\mu}}\log_2\max\left\{\frac{15\Delta_0}{\varepsilon}, 1\right\} + \frac{56\sqrt{\Omega }\|K\|}{\sqrt{\mu\Delta_0\nu\omega}}\cdot\max\left\{\sqrt{\frac{15\Delta_0}{\varepsilon}}, 1\right\}.
		\end{align}
	\qed\end{proof}

\vgap

By Theorem~\ref{thmAGSRS}, the total number of operator evaluations involving $K$ performed by
Algorithm \ref{algAGSRS} to compute an $\varepsilon$-solution of \eqref{eqSmoothProblem} can be bounded by
\begin{align}
	\cO\left(\sqrt{\frac{L}{\mu}}\log\frac{1}{\varepsilon}+\frac{\|K\|}{\sqrt{\varepsilon}}\right),
\end{align}
which matches with the best-known complexity result (e.g., \cite{nesterov2005excessive}). However, 
the total number of gradient evaluations of $\nabla f$ is now bounded by
\begin{align}
	\cO\left(\sqrt{\frac{L}{\mu}}\log\frac{1}{\varepsilon}\right),
\end{align}
which drastically improves existing results from $\cO(1/\sqrt{\varepsilon})$ to $\cO(\log(1/\varepsilon))$. 


\section{Numerical experiments}
\label{secNumerical}
In this section, we present some preliminary experimental results on the proposed AGS algorithm. The algorithms for all the experiments are implemented in MATLAB R2016a, running on a computer with 3.6 GHz Intel i7-4790 CPU and 32GB RAM. The parameters of Algorithm \ref{algAGS} are set to Corollary \ref{corAGS} and Proposition \ref{corAGSS} for solving
composite smooth problems and bilinear saddle point problems, respectively.
\subsection{Smooth optimization}
Our first experiment is conducted on a portfolio selection problem, which can be formulated as a quadratic programming problem
\begin{align}
	\label{eqMarkowitz}
	\min_{x\in\Delta^n}\phi(x):=x^T(A^T\cF A+\cD)x\ \st b^T x \ge \eta,
\end{align}
where $\Delta^n := \{ x \in \R^n | \sum_{i=1}^n x_i = 1, x_i \ge 0, i = 1, \ldots, n\}$.
The above quadratic programming problem describes minimum variance portfolio selection strategy in a market with $n$ trading assets and $m$ factors that drive the market. In particular, we are assuming a market return model (see \cite{goldfarb2003robust})
\begin{align}
q = b + A^Tf + \varepsilon,
\end{align}
where $q\in\R^n$ describes the random return with mean $b\in\R^n$, $f\in\R^m$ is a normally distributed vector with distribution $f\sim N(0,\cF)$ that describes the factors driving the market, $A\in\R^{m\times n}$ is the matrix of factor loadings of the $n$ assets, and $\varepsilon\sim N(0,\cD)$ is the random vector of residual returns. The return of portfolio $x$ now follows the distribution
\begin{align}
q^T x\sim N(b^T x, x^T(A^T\cF A+\cD)x),
\end{align}
and problem \eqref{eqMarkowitz} describes the objective of minimizing the risk (in terms of variance) while obtaining expected return of at least $\eta$. 
It can be easily seen that Problem \eqref{eqMarkowitz} is 
a special case of \eqref{eqSmoothProblem} with
\begin{align}
	& f(x) = x^T\cD x, h(x) = x^T(A^T\cF A)x,\ X=\Set{x\in\Delta^n|b^Tx\ge \eta}, 
	\\
	& M = \lambda_{max}(A^T\cF A), \text{ and } L = \lambda_{max}(\cD).
\end{align}
Here $\lambda_{max}(\cdot)$ denotes the maximum eigenvalue. It should be noted that in practice we have $m<n$ and the eigenvalues of $\cD$ are much smaller than that of $A^T\cF A$. Consequently, the computational cost for gradient evaluation of $\nabla f$ is more expensive than that of $\nabla h$, and the Lipschitz constants $L$ in \eqref{eqL} and $M$ in \eqref{eqM} satisfy $L<M$.

To generate the datasets for this experiment, first we fix $n=5000$, $\eta=1$, choose $m$ from $\left\{2^4,2^5,\ldots,2^9\right\}$, and generate $b$, $A$, and $\cF$ randomly. Here each component of $b$ is generated uniformly in $[0,5]$, each component of $A$ is generated uniformly in $[0,1]$, and $\cF=B^TB$ where $B$ is a $\lceil m/2\rceil\times m$ matrix whose components are generated from standard normal distribution. Then, we estimate $M$ from the relation $M = \lambda_{max}(A^T\cF A)$, choose $L$ from $\left\{2^{-2}M,2^{-3}M,\ldots,2^{-15}M\right\}$, and set $\cD=L(C^TC)/\lambda_{max}(C^TC)$, where $C$ is a $2500$ by $5000$ matrix whose components are generated from the standard normal distribution. With such setting of $\cD$, we have $\lambda_{max}(D)=L$. In order to demonstrate the efficiency of the AGS algorithm, we compare it with Nesterov's accelerated gradient method (NEST) in \cite{nesterov2004introductory}. For both algorithms, we set the prox-function to the entropy function $V(x,u) = \sum_{i=1}^{n}u^{(i)}\ln(u^{(i)}/x^{(i)})$, where $u^{(i)}$ denotes the $i$-th component of $u$.
Observe that the computational costs per iteration are different for NEST and AGS,  
since NEST evaluates both ${\nabla f}$ and ${\nabla h}$ in each iteration, while AGS can skip the evaluation of ${\nabla f}$ from time to time.
In order to have a fair comparison between these two algorithms, we first run NEST for $300$ iterations, and then AGS for the same amount of CPU time as NEST. 
In Figure \ref{figQPall}, we plot the ratio of the objective values obtained by NEST of AGS in this manner. If the ratio is less than 1 (indicated by red cross in the figure), then the performance of NEST is better than that of the AGS, and if the ratio is greater than 1 (indicated by blue round), then AGS outperforms NEST. 
It should be noted that we plot the ratio rather than the difference of the objective values obtained by these algorithms, mainly
because these objective values for difference instances are quite different. 
We can observe from Figure \ref{figQPall} that for most of the choices of $m$ and $L$, AGS outperforms NEST in terms of objective value. In particular, as $M/L$ increases and 
$m$ decreases, the difference 
on the performance of AGS and NEST becomes more and more significant. Therefore, we can conclude that AGS performs better than NEST when either the difference 
between $M$ and $L$ is larger or the computational cost for evaluating $\nabla h$ is cheaper. Such observations are consistent with our theoretical complexity analysis regarding AGS and NEST. 

\begin{figure}[!htbp]
	\centering
	\includegraphics[width=.7\linewidth]{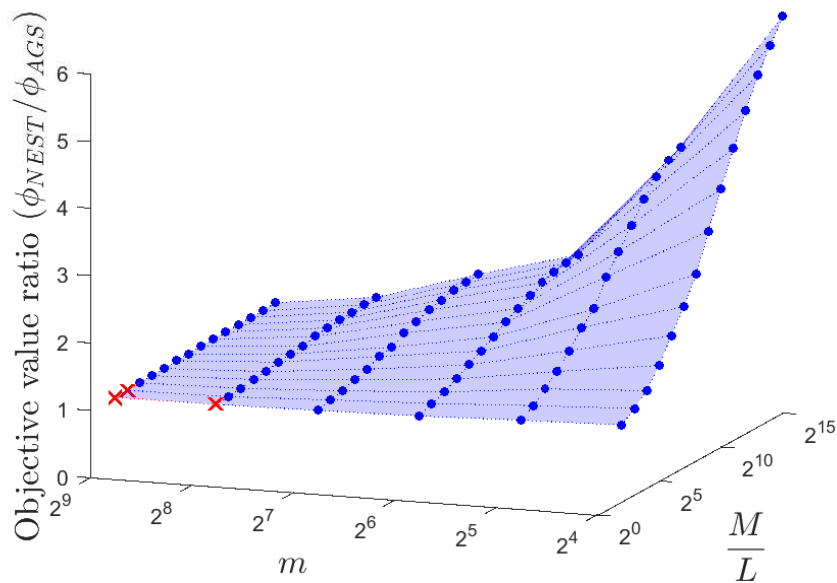}
	\caption{\label{figQPall} Ratio of objective values of AGS and NEST in terms of different choices of dimension $m$ and ratio $M/L$, after running the same amount of CPU time. Here $\phi_{AGS}$ and $\phi_{NEST}$ are the objective values corresponding to approximated solutions obtained by AGS and NEST within the same amount of CPU time. Cross markers in red imply that $\phi_{NEST}<\phi_{AGS}$, i.e., NEST outperforms AGS, while round markers in blue indicate that AGS outperforms NEST.}
\end{figure}

In addition to Figure \ref{figQPall}, we also report in Tables \ref{tabmAGS_Markowitz} and \ref{tabMLAGS_Markowitz} the numbers of gradient evaluations of $\nabla f$ and $\nabla h$ performed by
the AGS method with respect to different choices of dimension $m$ and ratio $M/L$.  As mentioned earlier, the numbers of gradient evaluations 
of $\nabla f$ and $\nabla h$ in 300 iterations of NEST are both 300. Several remarks are in place regarding the results obtained in these two tables. 
First, during the same amount of CPU time, the AGS is able to perform more gradient evaluations of $\nabla h$ by skipping
the computation of $\nabla f$. Noting that the Lipschitz constants of $\nabla h$ and $\nabla f$ satisfy $M>L$, by the complexity bounds \eqref{eqBoundLM}, the increased number of gradient evaluations of $\nabla h$ results in lower objective  function value of AGS. The advantage of AGS over NEST in terms of objective function value becomes more significant as the ratio $M/L$ increases.
Second, the lower computational cost we have for gradient evaluation of $\nabla h$, the more gradient evaluations of $\nabla h$ we can perform at each time when skipping gradient evaluation of $\nabla f$. Therefore, we can observe in Table \ref{tabmAGS_Markowitz} that the reduction of dimension $m$ leads to more evaluations of $\nabla h$, and more significant performance improvment
of AGS over NEST. 
Finally, it should be noted that AGS requires more evaluations of $\nabla h$ in order to obtain the same level of accuracy as NEST in terms of objective function value. In particular, we can observe from the case with $m=512$ in Table \ref{tabmAGS_Markowitz} and the case with $M/L=2^2$ in Table \ref{tabMLAGS_Markowitz} that AGS requires approximately triple amount of gradient evaluations of $\nabla h$ in order to obtain the same objective value as NEST. One plausible explanation is that the estimate of $M/L$ in Corollary \ref{corAGS} is conservative, resulting in 
a larger number of inner iterations $T_k$ in \eqref{eqcondAGSouter}. However, when $m$ is small or when the estimated ratio $M/L$ is high, AGS can perform much more numbers of gradient evaluations of $\nabla h$ to overcome the aforementioned disadvantage. It would be interesting to develop a scheme that provides more accurate estimate of the ratio $M/L$, possibly through some line search
procedures.

\begin{table}[!hbtp]
	\centering
	\caption{\label{tabmAGS_Markowitz}Numbers of gradient evaluations of $\nabla f$ and $\nabla h$ performed by the AGS method for solving \protect\eqref{eqMarkowitz} with $M/L=1024$, after running the same amount of CPU time as 300 iterations of NEST. Here $\phi_{AGS}$ and $\phi_{NEST}$ are the objective values corresponding to the approximated solutions obtained by AGS and NEST, respectively.}
	\begin{tabular}{|*{4}{c|}}
		\hline
		{$m$} & \# AGS evaluations of $\nabla f$ & \# AGS evaluations of $\nabla h$ & {$\ds{\phi_{NEST}}/{\phi_{AGS}}$}
		\\\hline
16 & 104 & 3743 & 382.5\%  
\\\hline
32 & 100 & 3599 & 278.6\%  
\\\hline
64 & 95 & 3419 & 183.3\%  
\\\hline
128 & 65 & 2339 & 152.8\%  
\\\hline
256 & 42 & 1499 & 120.1\%  
\\\hline
512 & 27 & 936 & 104.8\%  
\\\hline
	\end{tabular}
\ \\
\vgap
\vgap
	\renewcommand{\arraystretch}{1.1}
	\centering
	\caption{\label{tabMLAGS_Markowitz}Numbers of gradient evaluations of $\nabla f$ and $\nabla h$ performed by the AGS method for solving \protect\eqref{eqMarkowitz} with $m=64$, after running the same amount of CPU time as 300 iterations of NEST. Here $\phi_{AGS}$ and $\phi_{NEST}$ are the objective values corresponding to the approximated solutions obtained by AGS and NEST, respectively.}	
	\begin{tabular}{|*{4}{c|}}
		\hline
		{$M/L$} & \# AGS evaluations of $\nabla f$ & \# AGS evaluations of $\nabla h$ & {$\ds{\phi_{NEST}}/{\phi_{AGS}}$}
		\\\hline
$2^{15}$ & 23 & 4471 & 212.5\%  
\\\hline
$2^{14}$ & 31 & 4327 & 210.5\%  
\\\hline
$2^{13}$ & 41 & 4097 & 206.5\%  
\\\hline
$2^{12}$ & 57 & 4038 & 201.6\%  
\\\hline
$2^{11}$ & 72 & 3648 & 192.4\%  
\\\hline
$2^{10}$ & 95 & 3419 & 183.3\%  
\\\hline
$2^{9}$ & 114 & 2961 & 173.3\%  
\\\hline
$2^{8}$ & 143 & 2698 & 161.7\%  
\\\hline
$2^{7}$ & 164 & 2132 & 150.5\%  
\\\hline
$2^{6}$ & 186 & 1859 & 140.1\%  
\\\hline
$2^{5}$ & 210 & 1470 & 129.2\%  
\\\hline
$2^{4}$ & 225 & 1125 & 120.0\%  
\\\hline
$2^{3}$ & 258 & 1032 & 112.9\%  
\\\hline
$2^{2}$ & 253 & 759 & 104.5\%  
\\\hline
	\end{tabular}
\end{table}

\subsection{Image reconstruction}

In this subsection, we consider the following total-variation (TV) regularized image reconstruction problem:
\begin{align}
\label{eqTV}
\min_{x\in \R^n} \psi(x):= \frac{1}{2}\|Ax-b\|^2+ \eta\|Dx\|_{2,1}.
\end{align}
Here $x\in \R^n$ is the $n$-vector form of a two-dimensional image to be reconstructed, $\|Dx\|_{2,1}$ is the discrete form of the TV semi-norm where $D$ is the finite difference operator, $A$ is a measurement matrix describing the physics of data acquisition, and $b$ is the observed data.
It should be noted that problem \eqref{eqTV} is equivalent to
$$\min_{x\in \R^n}\frac{1}{2}\|Ax-b\|^2+ \max_{y\in Y}\eta \langle Dx, y \rangle ,$$
where $Y:=\{y\in\R^{2n}:\|y\|_{2,\infty}:=\max_{i=1,\ldots,n}\|(y^{(2i-1)}, y^{(2i)})^T\|_2\leq 1\}$. The above form can
be viewed as a special case of the bilinear SPP \eqref{eqSPP} with
\begin{align}
f(x):=\frac{1}{2}\|Ax-b\|^2, K:=\eta D, \text{ and }J(y) \equiv 0,
\end{align}
and the associated constants are $L=\lambda_{max}(A^TA)$ and $\|K\|=\eta\sqrt{8}$ (see, e.g., \cite{chambolle2004algorithm}). Therefore, as discussed in Section \ref{secAGSSPP}, such problem can be solved by AGS after incorporating the smoothing technique in \cite{nesterov2005smooth}.


In this experiment, the dimension of $A\in\R^{m\times n}$ is set to  $m=\lceil n/3\rceil$. Each component of $A$ is generated from a Bernoulli distribution, namely, it takes equal probability for the values $1/\sqrt{m}$ and $-1/\sqrt{m}$ respectively. We generate $b$ from a ground truth image $x_{true}$ with $b=Ax_{true}+\epsilon$, where $\epsilon\sim N(0,0.001I_n)$. Two ground truth images $x_{true}$ are used in the experiment, namely, the 256 by 256 ($n=65536$) image ``Cameraman'' and the 135 by 198 ($n=26730$) image ``Onion''. Both of them are built-in test images in the MATLAB image processing toolbox.
We compare the performance of AGS and NEST 
for each test image with different smoothing parameter $\rho$ in \eqref{eqhrho}, and TV regularization parameter $\eta$ in \eqref{eqTV}. For both algorithms, the prox-functions $V(x,u)$ and $W(y,v)$ are set to Euclidean distances $\|x-u\|_2^2/2$ and $\|y-v\|_2^2/2$ respectively.
In order to perform a fair comparison, we run NEST for 200 iterations first, and then run AGS with the same amount of CPU time. 

Tables \ref{tabAGSCameraman_LSTV} and \ref{tabAGSOnion_LSTV} show the comparison between AGS and NEST in terms of gradient evaluations of $\nabla f$, operator evaluations of $K$ and $K^T$, and objective values \eqref{eqTV}. It should be noted that in 200 iterations of the NEST algorithm, the number of gradient evaluations of $\nabla f$ and operator evaluations of $K$ and $K^T$ are 
given by 200 and 400, respectively. We can make a few observations about the results reported in these tables.  
First, by skipping gradient evaluations of $\nabla f$, AGS is able to perform more operator evaluation of $K$ and $K^T$ during the same amount of CPU time. Noting the complexity bounds \eqref{eqOptSmoothXY} and \eqref{eqOptSPPXY}, we can observe that the extra amount of operator evaluations $K$ and $K^T$ can possibly result in better approximate solutions obtained by CGS in terms of objective values. It should be noted that in problem \eqref{eqTV}, $A$ is a dense matrix while $D$ is a sparse matrix. Therefore, 
a very large number of extra evaluations of $K$ and $K^T$ can be performed for each skipped gradient evaluation of $\nabla f$.
Second, for the smooth approximation problem \eqref{eqSmoothApproxProblem}, the Lipschitz constant $M$ of $h_{\rho}$ is
given by $M=\|K\|^2/{\rho\omega}$. Therefore, for the cases with $\rho$ being fixed, larger values of $\rho$ result in larger norm $\|K\|$, and consequently larger Lipschitz constant $M$. Moreover, 
for the cases when $\eta$ is fixed, smaller values of $\rho$ also lead to larger Lipschitz constant $M$. For both cases, as the ratio of $M/L$ increases, 
we would skip more and more gradient evaluations of $\nabla f$, and allocate more CPU time for operator evaluations of $K$ and $K^T$, 
which results in more significant performance improvement of AGS over NEST. Such observations are also consistent with our previous theoretical 
complexity analysis regarding AGS and NEST for solving composite bilinear saddle point problems.

\begin{table}[H]
	\centering
	\caption{\label{tabAGSCameraman_LSTV}Numbers of gradient evaluations of $\nabla f$ and $\nabla h$ performed by the AGS method for solving \protect\eqref{eqTV} with ground truth image ``Cameraman'', after running the same amount of CPU time as 200 iterations of NEST. Here $\psi_{AGS}$ and $\psi_{NEST}$ are the objective values of \protect\eqref{eqTV} corresponding to approximated solutions obtained by AGS and NEST, respectively.}	
		\renewcommand{\arraystretch}{1.1}
	\begin{tabular}{|lc|*{4}{c|}}
		\hline
		\multicolumn{2}{|c|}{Problem} & \begin{minipage}{2.5cm} \# AGS evaluations of $\nabla f$\end{minipage} & \begin{minipage}{2.5cm}  \# AGS evaluations of $K$ and $K^T$ \end{minipage}& $\psi_{AGS}$ & $\psi_{NEST}$
		\\\hline
		$\eta = 1$,&$\rho=10^{-5}$ & 52 & 37416 & 723.8 & 8803.1
		\\\hline
		$\eta = 10^{-1}$,&$\rho=10^{-5}$ & 173 & 12728 & 183.2 & 2033.5
		\\\hline
		$\eta = 10^{-2}$,&$\rho=10^{-5}$ & 198 & 1970 & 27.2 & 38.3
		\\\hhline{|==*{4}{=|}}
		$\eta = 10^{-1}$,&$\rho=10^{-7}$ & 51 & 36514 & 190.2 & 8582.1
		\\\hline
		$\eta = 10^{-1}$,&$\rho=10^{-6}$ & 118 & 27100 & 183.2 & 6255.6
		\\\hline
		$\eta = 10^{-1}$,&$\rho=10^{-5}$ & 173 & 12728 & 183.2 & 2033.5
		\\\hline
		$\eta = 10^{-1}$,&$\rho=10^{-4}$ & 192 & 4586 & 183.8 & 267.2
		\\\hline
		$\eta = 10^{-1}$,&$\rho=10^{-3}$ & 201 & 2000 & 190.4 & 191.2
		\\\hline
		$\eta = 10^{-1}$,&$\rho=10^{-2}$ & 199 & 794 & 254.2 & 254.2
		\\\hline
	\end{tabular}
\ \\
\vgap
\vgap
	\centering
	\caption{\label{tabAGSOnion_LSTV}Numbers of gradient evaluations of $\nabla f$ and $\nabla h$ performed by the AGS method for solving \protect\eqref{eqTV}, after running the same amount of CPU time as 200 iterations of NEST. Here $\psi_{AGS}$ and $\psi_{NEST}$ are the objective values of \protect\eqref{eqTV} corresponding to approximated solutions obtained by AGS and NEST, respectively.}	
	\renewcommand{\arraystretch}{1.1}
	\begin{tabular}{|lc|*{4}{c|}}
		\hline
		\multicolumn{2}{|c|}{Problem} & \begin{minipage}{2.5cm} \# AGS evaluations of $\nabla f$\end{minipage} & \begin{minipage}{2.5cm}  \# AGS evaluations of $K$ and $K^T$ \end{minipage}& $\psi_{AGS}$ & $\psi_{NEST}$
		\\\hline
$\eta =1$,&$\rho=10^{-5}$ & 37 & 26312 & 295.6 & 2380.3
\\\hline
$\eta=10^{-1}$,&$\rho=10^{-5}$ & 149 & 10952 & 52.6 & 608.5
\\\hline
$\eta=10^{-2}$,&$\rho=10^{-5}$ & 193 & 1920 & 6.9 & 10.6
		\\\hhline{|==*{4}{=|}}
$\eta = 10^{-1}$,&$\rho=10^{-7}$ & 62 & 44344 & 52.9 & 2325.5
\\\hline
$\eta = 10^{-1}$,&$\rho=10^{-6}$ & 102 & 23380 & 52.6 & 1735.1
\\\hline
$\eta=10^{-1}$,&$\rho=10^{-5}$ & 149 & 10952 & 52.6 & 608.5
\\\hline
$\eta = 10^{-1}$,&$\rho=10^{-4}$ & 174 & 4154 & 52.8 & 70.0
\\\hline
$\eta = 10^{-1}$,&$\rho=10^{-3}$ & 192 & 1910 & 54.5 & 54.7
\\\hline
$\eta = 10^{-1}$,&$\rho=10^{-2}$ & 198 & 790 & 68.6 & 68.6
\\\hline
	\end{tabular}
\end{table}

\section{Conclusion}
\label{secConclusion}
We propose an accelerated gradient sliding (AGS) method for solving certain classes of structured convex optimization. The main feature of the proposed AGS method is that it could skip gradient computations of a smooth component in the objective function from time to time, while still maintaining the overall optimal 
rate of convergence for these probems. In particular, for minimizing the summation of two smooth convex functions, the AGS method can skip the gradient computation of the function with a
smaller Lipschitz constant, resulting in sharper complexity results than the best known so-far complexity bound under the traditional black-box assumption. Moreover, for solving a class of bilinear saddle-point problem, by applying the AGS algorithm to solve its smooth approximation, we show that the number of gradient evaluations of the smooth component may be reduced to $\cO(1/\sqrt{\varepsilon})$, which improves the previous $\cO(1/\varepsilon)$ complexity bound in the literature. More significant savings on gradient computations can be obtained when the objective function is strongly convex, with the number of gradient evaluations being reduced further to $\cO(\log(1/\varepsilon))$. Numerical experiments further confirm the potential advantages of these new optimization schemes
for solving structured convex optimization problems.

\bibliographystyle{myspmpsci}
\bibliography{yuyuan}{}

\end{document}